\documentclass{amsart}
\usepackage{fullpage}
\usepackage[utf8]{inputenc}
\usepackage{amsmath, graphicx}
\usepackage{color}
\usepackage{amsthm}
\usepackage{amssymb}
\usepackage{amscd}
\usepackage{mathtools, bbm}
\usepackage{amsfonts}  
\usepackage{hyperref}
\usepackage{subcaption}
\usepackage{tikz}
\usepackage{comment}
\usepackage{cleveref}
\usepackage{relsize}
\usepackage{enumitem} 
\usetikzlibrary{patterns}
\tikzset{
  norm/.style     = {shape=circle, draw},
  blue/.style     = {shape=circle, draw, fill=blue!25},
  high/.style     = {shape=circle, draw, color=red},
  bluehigh/.style = {shape=circle, draw, color=red, fill=blue!25},
  red/.style      = {shape=circle, draw, fill=red!25},
  both/.style     = {shape=circle, draw, fill=violet!35},
  root/.style     = {node, bottom color=red!30},
  env/.style      = {treenode, font=\ttfamily\normalsize},
  dummy/.style    = {circle}
}
\tikzstyle{standard}=[circle, draw=black, fill=white, very thick, minimum size=7mm]
\tikzstyle{standard2}=[circle, draw=black, fill=white, very thick]
\tikzstyle{blue2}=[circle, draw=black, fill=blue!25, very thick]
\tikzstyle{small}=[circle, draw=black, fill=black, very thick, minimum size=4mm]
\tikzstyle{special}=[circle, draw=red!60, fill=red!5, very thick, minimum size=5mm]
\newtheorem{theorem}{Theorem}[section]
\newtheorem{lemma}[theorem]{Lemma}
\newtheorem{cor}[theorem]{Corollary}
\newtheorem{prop}[theorem]{Proposition}
\theoremstyle{definition}
\newtheorem{df}[theorem]{Definition}
\newtheorem{rem}[theorem]{Remark}

\newtheorem{ex}[theorem]{Example}

\newtheorem{conj}[theorem]{Conjecture}

\DeclareMathOperator{\F}{\mathcal{F}}
\DeclareMathOperator{\susp}{susp}

\DeclareMathOperator{\bbS}{\mathbb{S}} 
\DeclareMathOperator{\sgn}{\mathrm{sgn}} 
\DeclareMathOperator{\tr}{\mathrm{tr}} 

\newcommand{\sheila}[1]{{\color{purple} \sf Sheila: [#1]}}

\newcommand{\heha}{\tikz[baseline=.1ex, scale=.5]{
    \draw[fill=black] (0,1) circle (3pt);
    \draw[fill=black] (1,1) circle (3pt);
    \draw[fill=black] (1,0) circle (3pt);
    \draw[fill=black] (2,0) circle (3pt);
    \draw (0,1) -- (1,1);
    \draw (1,1) -- (1,0); 
    \draw (1,0) -- (2,0);
    } }

\newcommand{\hehb}{\tikz[baseline=.1ex, scale=.5]{
    \draw[fill=black] (0,1) circle (3pt);
    \draw[fill=black] (1,1) circle (3pt);
    \draw[fill=black] (1,0) circle (3pt);
    \draw[fill=black] (2,1) circle (3pt);
    \draw (0,1) -- (1,1);
    \draw (1,1) -- (1,0); 
    \draw (1,1) -- (2,1);
    } }

\newcommand{\hehc}{\tikz[baseline=.1ex, scale=.5]{
    \draw[fill=black] (0,0) circle (3pt);
    \draw[fill=black] (1,0) circle (3pt);
    \draw[fill=black] (1,1) circle (3pt);
    \draw[fill=black] (2,1) circle (3pt);
    \draw (0,0) -- (1,0);
    \draw (1,0) -- (1,1); 
    \draw (1,1) -- (2,1);
    } }

\newcommand{\hehd}{ \tikz[baseline=.1ex, scale=.5]{
    \draw[fill=black] (0,0) circle (3pt);
    \draw[fill=black] (1,1) circle (3pt);
    \draw[fill=black] (1,0) circle (3pt);
    \draw[fill=black] (2,0) circle (3pt);
    \draw (0,0) -- (1,0);
    \draw (1,1) -- (1,0); 
    \draw (1,0) -- (2,0);
    }}

\newcommand{\hevDOWN}{\tikz[baseline=.1ex, scale=.5]{
    \draw[fill=black] (0,1) circle (3pt);
    \draw[fill=black] (1,1) circle (3pt);
    \draw[fill=black] (2,1) circle (3pt);
    \draw[fill=black] (2,0) circle (3pt);
    \draw (0,1) -- (1,1);
    \draw (1,1) -- (2,1); 
    \draw (2,1) -- (2,0);
    } }

\newcommand{\vDOWNeh}{\tikz[baseline=.1ex, scale=.5]{
    \draw[fill=black] (0,0) circle (3pt);
    \draw[fill=black] (0,1) circle (3pt);
    \draw[fill=black] (1,1) circle (3pt);
    \draw[fill=black] (2,1) circle (3pt);
    \draw (0,0) -- (0,1) -- (1,1) -- (2,1);
    } }

\newcommand{\hevUP}{\tikz[baseline=.1ex, scale=.5]{
    \draw[fill=black] (0,0) circle (3pt);
    \draw[fill=black] (1,0) circle (3pt);
    \draw[fill=black] (2,0) circle (3pt);
    \draw[fill=black] (2,1) circle (3pt);
    \draw (0,0) -- (1,0) -- (2,0) -- (2,1);
    } }

\newcommand{\vUPeh}{\tikz[baseline=.1ex, scale=.5]{
    \draw[fill=black] (0,1) circle (3pt);
    \draw[fill=black] (0,0) circle (3pt);
    \draw[fill=black] (1,0) circle (3pt);
    \draw[fill=black] (2,0) circle (3pt);
    \draw (0,1) -- (0,0) -- (1,0) -- (2,0);
    } }
    
\newcommand{\veva}{\tikz[baseline=.1ex, scale=.5]{
    \draw[fill=black] (0,-1) circle (3pt);
    \draw[fill=black] (0,0) circle (3pt);
    \draw[fill=black] (0,1) circle (3pt);
    \draw[fill=black] (1,0) circle (3pt);
    \draw (0,-1) -- (0,0) -- (0,1) ; 
    \draw (0,0) -- (1,0);
    } }

\newcommand{\vevb}{\tikz[baseline=.1ex, scale=.5]{
    \draw[fill=black] (0,-1) circle (3pt);
    \draw[fill=black] (0,0) circle (3pt);
    \draw[fill=black] (0,1) circle (3pt);
    \draw[fill=black] (-1,0) circle (3pt);
    \draw (0,-1) -- (0,0) -- (0,1) ; 
    \draw (0,0) -- (-1,0);
    } }

\newcommand{\vevc}{\tikz[baseline=.1ex, scale=.5]{
    \draw[fill=black] (0,0) circle (3pt);
    \draw[fill=black] (0,1) circle (3pt);
    \draw[fill=black] (1,0) circle (3pt);
    \draw[fill=black] (1,-1) circle (3pt);
    \draw (0,1) -- (0,0) -- (1,0) -- (1,-1); 
        } }

\newcommand{\vevd}{\tikz[baseline=.1ex, scale=.5]{
    \draw[fill=black] (0,-1) circle (3pt);
    \draw[fill=black] (0,0) circle (3pt);
    \draw[fill=black] (1,0) circle (3pt);
    \draw[fill=black] (1,1) circle (3pt);
    \draw (0,-1) -- (0,0) -- (1,0) -- (1,1); 
        } }

\newcommand{\hLEFTevDOWN}{\tikz[baseline=.1ex, scale=.5]{
    \draw[fill=black] (-1,1) circle (3pt);
    \draw[fill=black] (0,1) circle (3pt);
    \draw[fill=black] (0,0) circle (3pt);
    \draw[fill=black] (0,-1) circle (3pt);
    \draw (-1,1)--(0,1)--(0,0)--(0,-1); 
        } }

\newcommand{\hRIGHTevDOWN}{\tikz[baseline=.1ex, scale=.5]{
    \draw[fill=black] (1,1) circle (3pt);
    \draw[fill=black] (0,1) circle (3pt);
    \draw[fill=black] (0,0) circle (3pt);
    \draw[fill=black] (0,-1) circle (3pt);
    \draw (1,1)--(0,1)--(0,0)--(0,-1); 
        } }

\newcommand{\hhh}{\tikz[baseline=.1ex, scale=.5]{
    \draw[fill=black] (0,0) circle (3pt);
    \draw[fill=black] (1,0) circle (3pt);
    \draw[fill=black] (2,0) circle (3pt);
    \draw (0,0)--(1,0) -- (2,0); 
        } }    

\title{Topology of Cut complexes II}

\author[M. Bayer]{Margaret Bayer}
\address{Margaret Bayer: University of Kansas, Lawrence, Kansas, USA}
\email{bayer@ku.edu}

\author[M. Denker]{Mark Denker}
\address{Mark Denker: University of Kansas, Lawrence, Kansas, USA}
\email{mark.denker@ku.edu}

\author[M. Jeli\'c Milutinovi\'c]{Marija Jeli\'c Milutinovi\'c}
\address{Marija Jeli\'c Milutinovi\'c: University of Belgrade, Serbia}
\email{marija.jelic@matf.bg.ac.rs}

\author[S. Sundaram]{Sheila Sundaram}
\address{Sheila Sundaram: University of Minnesota, Minneapolis, Minnesota, USA}
\email{shsund@umn.edu}

\author[L. Xue]{Lei Xue}
\address{Lei Xue: University of Michigan, Ann Arbor, Michigan, USA}
\email{leixue@umich.edu}

\begin{document}    


\begin{abstract}
 We continue the study of the $k$-cut complex $\Delta_k(G)$ of a graph $G$ initiated in the paper of Bayer, Denker, Jeli\'c Milutinovi\'c, Rowlands, Sundaram and Xue [Topology of cut complexes of graphs, SIAM J. on Discrete Math. 38(2): 1630--1675 (2024)]. 
 We give explicit formulas for the $f$- and $h$-polynomials of the cut complex $\Delta_k(G_1+G_2) $ of the disjoint union of two graphs $G_1$ and $G_2$, and for the homology representation of $\Delta_k(K_m+K_n)$.
 We also study the cut complex of the squared path and the grid graph.  Our techniques include tools from combinatorial topology, discrete Morse theory and equivariant poset topology.
\end{abstract}
\maketitle
\tableofcontents
\section{Introduction}\label{sec:Intro}

This paper continues the investigations begun in \cite{BDJRSX-TOTAL2024} and  \cite{BDJRSX2024}.  We study the \emph{$k$-cut complex} $\Delta_k(G)$ of a graph $G=(V,E)$. This is the simplicial complex whose facets $\sigma$ are subsets of the vertex set $V$ of $G$, with  the property that the complement $V\setminus \sigma$ is a $k$-subset of $V$ inducing a disconnected subgraph of $G$. Our definition is inspired by a celebrated result of Fr\"oberg \cite{Froberg1990} for the case $k=2$,  which characterizes the Stanley-Reisner ideal generated by monomials of degree two in terms of chordal graphs. Consideration of higher degree monomials led us to 
the above generalization of the simplicial complex  to arbitrary $k$.  For more details see \cite{BDJRSX-TOTAL2024} and  \cite{BDJRSX2024}.

The behavior of the cut complex in relation to some well-known graph operations was established in \cite{BDJRSX2024}, and some particular families of graphs were studied.  Here we study in more detail the effect of one particular graph operation, namely, the disjoint union  $G_1+G_2$ of two graphs $G_1, G_2$. We show that in this case the topology of the $k$-cut complex of the disjoint union is determined by the join $G_1*G_2$ and the disjoint union of the complete graphs on the vertex sets of $G_1$ and $G_2$.  This observation then allows us to give precise formulas for the $f$- and $h$-polynomials of $\Delta_k(G_1+G_2)$. 
We also study two additional families of graphs, the squared path and the grid graph.  

The paper is organized as follows.  Section 2 recalls the necessary definitions from~\cite{BDJRSX2024}, background information from topology, and  facts about the face lattice of the $k$-cut complex from \cite{BDJRSX2024}, establishing further results that will be necessary for the study of the grid graph in Section~\ref{sec:Grid2021Sept15-17}.

In Section 3 we present an in-depth discussion of the effect of the disjoint union $G_1+G_2$ of two graphs $G_1, G_2$ on the cut complex $\Delta_k$. It had already been shown in \cite{BDJRSX2024} that 
$\Delta_k(G_1+G_2)$ is shellable if and only if $\Delta_k(G_i)$ is shellable for $i=1,2$.  We give simple explicit formulas for the $f$-polynomial and the $h$-polynomial of $\Delta_k(G_1+G_2)$ in terms of the respective polynomials for $\Delta_k(G_i)$, $i=1,2$; see Theorem~\ref{fdisjunion} and Corollary~\ref{hdisjunion}.  The $k$-cut complex $\Delta_k(K_{n_1}+K_{n_2})$, where $n_i$ is the number of vertices of the graph $G_i$ and $K_{n_i}$ is the complete graph, plays a crucial role in determining the topology of $\Delta_k(G_1+G_2)$. 

Section 4 studies the   $k$-cut complex of the squared path, showing that it is always shellable. We give an enumeratively interesting conjecture about the Betti numbers.

Section~\ref{sec:Grid2021Sept15-17} examines the grid graph. The homotopy type of the  $2$-cut complex was determined in \cite{BDJRSX-TOTAL2024}.  Here we show that for $k=3,4$, the $k$-cut complex has the homotopy type of a wedge of spheres in the top dimension, and we determine the Betti numbers. There is a considerable literature on independence complexes of grid graphs, 
and the topology in general is difficult to determine.  The case of the cut complex seems to be similarly challenging.  We prove shellability of the 3-cut complex, and determine the homotopy type of the 4-cut complex by a discrete Morse matching. For the 6-cut complex we show that homology is torsion-free and occurs in at most the top two dimensions.

Section~\ref{sec:homology-reps} studies group actions on the homology of the cut complexes. We present a formula for the representation of the symmetric group on the unique nonvanishing homology of $\Delta_k(K_{n_1}+K_{n_2})$. As shown in \cite{BDJRSX2024} for the edgeless graph and the complete multipartite graph, the Specht modules indexed by hooks play a prominent role in the representation.  Also included here are descriptions of the homology representation on the cut complexes of the path and the cycle for their respective symmetry groups.

The paper concludes with ideas for further exploration.

\section{Definitions}\label{sec:Definitions}
\phantom{}
General references for simplicial complexes, shellability and topology are \cite{BjTopMeth1995}, \cite{Hatcher2002} 
and \cite{WachsPosetTop2007}, and \cite{WestGraphTheory1996} for graph theory.  All graphs in this paper are simple (no loops and no multiple edges).

\begin{df}\label{def:simplicial-complex}  A \emph{simplicial complex} $\Delta$ on a set $A$ is a collection of subsets of $A$ 
\[ \sigma\in \Delta \text{ and } \tau\subseteq \sigma \Rightarrow \tau \in \Delta. \] 
The elements of $\Delta$ are called its \emph{faces} or \emph{simplices}. 
If the collection of subsets is empty, i.e., $\Delta$ has no faces,  we call $\Delta$ the \emph{void complex}.  Otherwise $\Delta$ always contains the empty set as a face. 

The \emph{dimension of a face} $\sigma$, $\dim(\sigma)$, is one less than its cardinality; thus the dimension of the empty face is $(-1)$, and the 0-dimensional faces are the \emph{vertices} of $\Delta$. A \emph{$d$-face} or \emph{$d$-simplex} is a face of dimension $d$.  The maximal faces of $\Delta$ are called its \emph{facets}, and the maximum dimension of a facet is the \emph{dimension} $\dim(\Delta)$ of the nonvoid simplicial complex $\Delta$. We  write $\Delta=\langle\mathcal{F}\rangle$ for the simplicial complex $\Delta$ whose set of facets is $\mathcal{F}$.

In this paper all simplicial complexes will be \emph{finite}, that is, the vertex set is finite.

A (nonvoid) simplicial complex is \emph{pure} if all its facets have the same dimension, which is then the dimension of the complex.

The \emph{join} of two simplicial complexes $\Delta_1$ and $\Delta_2$ with disjoint vertex sets is the complex \begin{center}{$\Delta_1 * \Delta_2= \{\sigma\cup \tau: \sigma\in \Delta_1, \tau\in \Delta_2\}.$}\end{center}
Thus the join $\Delta_1 * \Delta_2$ contains $\Delta_1$ and $ \Delta_2$ as subcomplexes.

The \emph{cone} over $\Delta$ and the \emph{suspension} of $\Delta$ are the complexes 
\begin{center}{$\mathrm{cone}(\Delta)=\Delta* \Gamma_1, \ 
\mathrm{susp}(\Delta)=\Delta*\Gamma_2=\Delta*\{u\}\cup\Delta*\{v\},$}\end{center}
where $\Gamma_1$ is the 0-dimensional simplicial complex with one vertex, and $\Gamma_2$ is the 0-dimensional complex with two vertices $u,v$. 

The \emph{$j$-skeleton} 
of $\Delta$ is the simplicial complex formed by the faces of $\Delta$ of dimension at most $j$.

\end{df}

\begin{df}[{\cite[Chapter III, Section 2]{RPSCCA1996}, \cite[Section~11.2]{BjTopMeth1995}}]\label{def:shelling}
An ordering $F_1,F_2,\dots,F_t$ of the facets of a simplicial complex $\Delta$ is a \emph{shelling} if, for every $j$ with $1<j\leq t$, $\left( \bigcup_{i=1}^{j-1}\langle F_i\rangle\right)\cap \langle F_j\rangle$
is a simplicial complex whose facets all have cardinality $|F_j|-1$, where $\langle F_i\rangle$ is the simplex generated by the face $F_i$.
The simplicial complex $\Delta$ is  \emph{shellable}  if it admits a shelling order of its facets. 

Equivalently, an ordering $F_1,F_2,\dots,F_t$ of the facets of $\Delta$ is a shelling if and only if for all $i,j$ such that $1\leq i<j\leq t$, there exists $k<j$ such that
$$F_i\cap F_j\subset F_k\cap F_j \quad \text{and} \quad |F_k\cap F_j| = |F_j|-1.$$

Given a shelling order $F_1,F_2,\dots,F_t$ of $\Delta$, the \emph{restriction} $r(F_i)$ of the facet $F_i$ is defined to be the unique minimal face of $\left( \bigcup_{j=1}^{i}\langle F_j\rangle\right)\setminus  \left( \bigcup_{j=1}^{i-1}\langle F_j\rangle\right)$. Hence the shelling order gives a partition into disjoint Boolean intervals 
$\Delta=\sqcup_{j=1}^t [r(F_j), F_j].$ The facet $F_i$ is called a \emph{full-restriction facet} if $r(F_i) = F_i$. 
\end{df}
\begin{rem}\label{rem:empty-complex-0-dim-complex} By convention, the void complex is shellable. The complex whose only face is the empty set 
 is vacuously shellable. 
The  complex with a unique nonempty facet is (also vacuously) shellable, and contractible. 
\end{rem}
See Figures~\ref{fig:Ex-Delta2}(\textsc{b}) and~\ref{fig:Ex-Delta4}(\textsc{b}) for examples of shellable and nonshellable complexes.

In combinatorial topology, shellability is an important tool  for determining the homotopy type of simplicial complexes, thanks to the following theorem of Bj\"orner.
\begin{theorem}[{\cite[(9.19) and Sec. 11]{BjTopMeth1995}}] \label{thm:shell-implies-homotopytype}
A pure shellable simplicial complex of dimension $d$ has the homotopy type of a wedge of spheres, all of dimension $d$. (We include in this the wedge of no spheres, that is, when the complex is contractible.)
The number of spheres in the wedge is the number of full restriction facets in the shelling.
\end{theorem}

\begin{df}\label{def:graph-disc-sep}
Let $G=(V,E)$ be a graph. 
If $S$ is a subset of the vertex set $V$,  write $G[S]$ to denote the \emph{induced subgraph} of $G$ whose vertex set is $S$. 
\end{df}

\begin{df}\label{def:cut-cplx}
Let $G=(V,E)$ be a graph on $|V|=n$ vertices, and let $k\ge 2$. Define the \emph{$k$-cut complex} of the graph $G$  to be the $(n-k-1)$-dimensional simplicial complex 
\[\Delta_k(G):=\langle F\subseteq V, |F|=n-k\mid G[V\setminus F] \text{ is disconnected}\rangle.\] The facets of the cut complex $\Delta_k(G)$ are the vertex subsets of $G$ of size $(n-k)$ whose removal disconnects the graph $G$.   Thus $\sigma$ is a face of the cut complex $\Delta_k(G)$ if and only if its complement $V\setminus \sigma$ contains a subset $S$ of size $k$ such that the induced subgraph $G[S]$ is disconnected.   Note  the inclusion
$\Delta_{k+1}(G)\subseteq \Delta_k(G)$ for $ k\ge 2,$ and the fact that the vertices of $\Delta_k(G)$ may be  a proper subset of the vertices of the graph $G$. See Figures~\ref{fig:Ex-Delta2} and~\ref{fig:Ex-Delta4} for contrasting examples, reproduced here from \cite{BDJRSX2024} for completeness.
\end{df}
    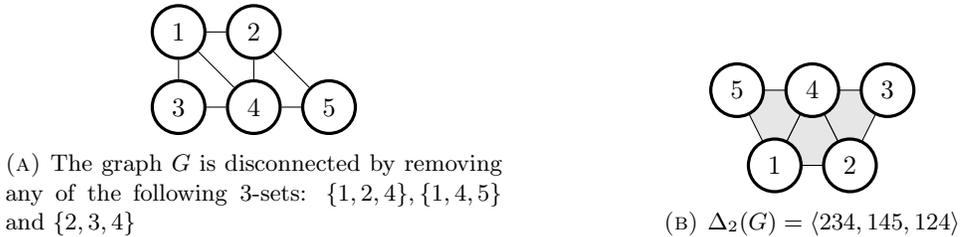
\begin{figure}[htb]
        \centering
        \begin{subfigure}{0.4\textwidth}
        \centering
        \begin{tikzpicture}[scale=0.5]
\node[standard] (node3) at (0,0) {3};
\node[standard] (node1) at (0,2) {1};
\node[standard] (node4) at (2,0) {4};
\node[standard] (node2) at (2,2) {2};
\node[standard] (node5) at (4,0) {5};

\draw (node1) -- (node3);
\draw (node1) -- (node2);
\draw (node1) -- (node4);
\draw (node2) -- (node4);
\draw (node2) -- (node5);
\draw (node3) -- (node4);
\draw (node4) -- (node5);
        \end{tikzpicture}
        \caption{The graph $G$ is disconnected by removing any of the following  3-sets: $\{1,2,4\}, \{1,4,5\}$ and $\{2,3,4\}$}
        \end{subfigure} \qquad
        \begin{subfigure}{0.4\textwidth}
        \centering
        \begin{tikzpicture}[scale=0.5]
\draw[fill=gray!20] (3,0) -- (4,2) -- (2,2) -- cycle;
\draw[fill=gray!20] (1,0) -- (2,2) -- (3,0) -- cycle;
\draw[fill=gray!20] (1,0) -- (2,2) -- (0,2) -- cycle;

\node[standard] (node5) at (0,2) {5};
\node[standard] (node1) at (1,0) {1};
\node[standard] (node4) at (2,2) {4};
\node[standard] (node2) at (3,0) {2};
\node[standard] (node3) at (4,2) {3};
        \end{tikzpicture}
       \caption{$\Delta_2 (G)=\langle234,145,124\rangle$}
        \end{subfigure}
       \caption{(Shellable) 2-cut complex of graph $G$}\label{fig:Ex-Delta2}
    \end{figure}

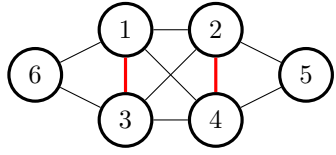
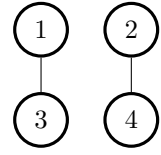
\begin{figure}[htb!]
        \centering
        \begin{subfigure}{0.4\textwidth}
        \centering
        \begin{tikzpicture}[scale=0.6]
\node[standard] (node3) at (0,0) {3};
\node[standard] (node1) at (0,2) {1};
\node[standard] (node4) at (2,0) {4};
\node[standard] (node2) at (2,2) {2};
\node[standard] (node5) at (4,1) {5};
\node[standard] (node6) at (-2,1) {6};
\draw[very thick, color=red] (node1) -- (node3); 
\draw (node1) -- (node2);
\draw (node1) -- (node4);  
\draw (node2) -- (node3);  
\draw[very thick, color=red] (node2) -- (node4);
\draw (node2) -- (node5);
\draw (node3) -- (node4);
\draw (node4) -- (node5);
\draw (node1) -- (node6);
\draw (node3) -- (node6);

        \end{tikzpicture}
        \caption{$G$ is disconnected by removing\\ one of the 2-sets $\{1,3\}$, $\{2,4\}$}
        \end{subfigure} \qquad
        \begin{subfigure}{0.4\textwidth}
        \centering
        \begin{tikzpicture}[scale=0.6]
\node[standard] (node3) at (0,0) {3};
\node[standard] (node1) at (0,2) {1};
\node[standard] (node4) at (2,0) {4};
\node[standard] (node2) at (2,2) {2};

\draw (node1) -- (node3);
\draw (node2) -- (node4);
        \end{tikzpicture}
        \caption{$\Delta_4 (G)=\langle 13, 24\rangle$}
        \end{subfigure}
       \caption{(Nonshellable) 4-cut complex of graph $G$}\label{fig:Ex-Delta4}
    \end{figure}

In \cite{BDJRSX-TOTAL2024} we considered a related complex, the {\em total $k$-cut complex}.  The facets of that complex are complements of independent sets of $k$ vertices, that is, those vertex sets that induce ``totally disconnected'' subgraphs.  Clearly, for $k=2$ these are the same as those of \Cref{def:cut-cplx}, so for $2$-cut complexes we will sometimes use results from \cite{BDJRSX-TOTAL2024}.

\begin{ex}\label{ex:Examples-cut-complex}  Let $G$ be a graph on $n$ vertices. We record some easy facts about cut complexes.  
\begin{enumerate}

\item $\Delta_k(G)$ is void if $k=1$ or $k>n$. 

\item $\Delta_n(G)$ is 
$\begin{cases} \text{ the void complex, } & \text{if $G$ is connected},\\
            \text{ the $(-1)$-dimensional complex } \{\emptyset\},  
                              & \text{otherwise}. \end{cases}$

\item $\Delta_k(G)$ is void for $n-k\le r-1$ if $G$ is $r$-connected, since at least $r$ vertices must be removed to disconnect the graph.

\item  If $G$ is the complete graph $K_n$, then $\Delta_k(G)$  is void for all $k\ge 1$. 

\item If $G=E_n$ is the edgeless graph on $n$ vertices, then for $2\le k\le n-1$, $\Delta_k(G)$ is the $(n-k-1)$-skeleton of an $(n-1)$-dimensional simplex, hence shellable \cite{BjWachsI1996}.
\end{enumerate}
\end{ex}

In view of Item (1) above,  we will assume $n\ge k\ge 2$ for the cut complex $\Delta_k(G)$.

 Recall that the Betti numbers of a topological space $X$ are the ranks of the homology groups of $X$.  In this paper we always consider reduced homology $\tilde{H}(X)$, with coefficients in the rationals. 
We are particularly interested in the case where $\Delta_k(G)$ is shellable and not contractible.  In that case, $\Delta_k(G)$ is homotopy equivalent to a wedge of spheres of dimension $n-k-1$, and the number of spheres in the wedge is the unique nonzero Betti number of  $\Delta_k(G)$. 
\begin{df}[{\cite{WestGraphTheory1996}}] \label{defn:chordal}
    A graph is \emph{chordal} if it has no induced cycle of size greater than 3.
    \end{df}
\begin{theorem}\label{thm:Froberg-Eagon-Reiner}
\cite{Froberg1990, EagonReiner1998}
The graph  $G$ is chordal  $\!\iff\!$ $\Delta_2(G)$ is shellable $\!\iff\!$ $\Delta_2(G)$ is vertex decomposable.
\end{theorem}
Figure~\ref{fig:Ex-Delta2} shows a chordal graph and its shellable 2-cut complex, whereas Figure~\ref{fig:Delta2C5} is an example of a nonchordal graph, with  nonshellable 2-cut complex.
These examples, reproduced for completeness from \cite{BDJRSX2024}, also illustrate  Fr\"oberg's Theorem, the first equivalence in~\Cref{thm:Froberg-Eagon-Reiner}.
\begin{figure}[htb]
\centering
\begin{subfigure}{0.4\textwidth}
\centering
\begin{tikzpicture}
\node[standard] (node1) at (0,3) {1};
\node[special] (node1) at (0,3) {1};
\node[standard] (node2) at (1.5,1.7) {2};
\node[special] (node2) at (1.5,1.7) {2};
\node[standard] (node3) at (1,0) {3};
\node[special] (node3) at (1,0) {3};
\node[standard] (node4) at (-1,0) {4};
\node[special] (node4) at (-1,0) {4};
\node[standard] (node5) at (-1.5,1.7) {5};
\node[special] (node5) at (-1.5,1.7) {5};
%
\draw (node1) -- (node2);
\draw (node2) -- (node3);
\draw (node3) -- (node4);
\draw (node4) -- (node5);
\draw (node5) -- (node1);
\end{tikzpicture}
\caption{The cycle $C_5$}
\end{subfigure}
\begin{subfigure}{0.4\textwidth}
\centering
\begin{tikzpicture}
\draw[fill=gray!20] (0,0) -- (0.7,1.5) -- (1.4,0) -- (0,0);
\draw[fill=gray!20] (2.1,1.5) -- (0.7,1.5) -- (1.4,0) -- (2.1,1.5);
\draw[fill=gray!20] (2.1,1.5) -- (1.4,0) -- (2.8,0) -- (2.1,1.5);
\draw[fill=gray!20] (3.5,1.5) -- (2.8,0) -- (2.1,1.5) -- (3.5,1.5);
\draw[fill=gray!20] (3.5,1.5) -- (4.2,0) -- (2.8,0) -- (3.5,1.5);
\draw[very thick, color=red,->] (0,0) -- (.55*0.7,.55*1.5);
\draw[very thick, color=red] (.55*0.7,.55*1.5) -- (0.7,1.5);
\draw[very thick, color=red,->] (3.5,1.5) -- (3.5+.55*.7,1.5-.55*1.5);
\draw[very thick, color=red] (3.5+.55*.7,1.5-.55*1.5) -- (4.2,0);

\node[standard] (n5) at (0,0) {5};
\node[standard] (n2) at (0.7,1.5) {2};
\node[standard] (n4) at (1.4,0) {4};
\node[standard] (n1) at (2.1,1.5) {1};
\node[standard] (n3) at (2.8,0) {3};
\node[standard] (n5') at (3.5,1.5) {5};
\node[standard] (n2') at (4.2,0) {2};
%
\end{tikzpicture}
\caption{$\Delta_2 (C_5)=\langle{245}{,124}{,134} {,135}{,235}\rangle$}
\end{subfigure}
\caption{The 2-cut complex for $C_5$ is a M\"obius strip}
\label{fig:Delta2C5}
\end{figure}
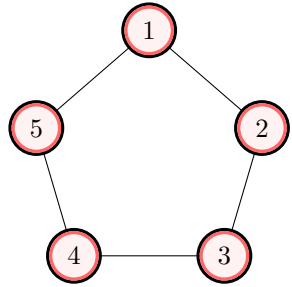
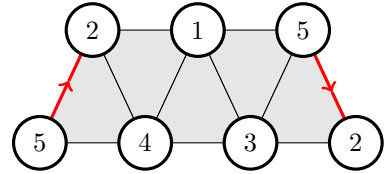

\subsection{Some background from  combinatorial topology}\label{sec:background}
\phantom{}

The following propositions are useful in analyzing cut complexes.  

\begin{prop}[{\cite[Chapter 21, (21.3)]{VINK2008}, \cite[Proposition~2.7]{BDJRSX-TOTAL2024}}] \label{prop:TopFact2ndIsoThm}
Let $X$ be a topological space with subspaces $A,B \subseteq X$ such that $X=A\cup B$, $A\cap B\neq \emptyset$, and $A, B$ are both closed subspaces or both open subspaces. Then the quotient map $A/(A\cap B) \rightarrow X/B$ of the inclusion $A\hookrightarrow X$ is a homeomorphism.
\end{prop}

\begin{prop}[{\cite[Proposition~0.17, Example~0.14]{Hatcher2002}}] \label{prop:quotient-by-contractible-homotopy}
Let $(X,A)$ be a CW pair consisting of a CW complex $X$ and a subcomplex $A$. 
\begin{enumerate}
\item If the subcomplex $A$ is contractible, then the quotient map $X\rightarrow X/A$ is a homotopy equivalence.
\item If $A$ is contractible in the complex $X$, then there is a homotopy equivalence  
\[X/A\simeq X \vee \susp(A).\]
\end{enumerate} 
\end{prop}

\begin{prop}\label{prop:complete2-skeleton}
    Let $X$ be a simplicial complex such that its 2-skeleton is complete. Then $X$ is simply connected.
\end{prop}
\begin{proof}
    The proof follows directly from  Proposition 4.12 in \cite{Hatcher2002}. This proposition claims that for arbitrary $n$, the skeleton $X^n$ and the entire complex $X$ have isomorphic homotopy groups $\pi_i$ for all $i < n$. Therefore $\pi_1(X) \cong \pi_1(X^2)$. Further, since $X^2$ contains all triangles on the vertex set $V(X)$ of $X$, we know that $X^2$ is the same as $Y^2$, where $Y$ is the full simplex on $V(X)$ vertices. We apply Proposition 4.12 from \cite{Hatcher2002} to $Y$ and we obtain:
    $$\pi_1(X^2) \cong \pi_1(Y^2) \cong \pi_1(Y) \cong 0.$$
    The triviality of the last group follows from the fact that every simplex is contractible, hence it has all trivial homotopy groups. 
    The connectivity of $X$ follows from the connectivity of $X^2$, so we conclude that the complex $X$ is simply connected. 
\end{proof}

The converse of \Cref{prop:complete2-skeleton} does not hold, as shown by the example in Figure~\ref{fig:Ex-Delta2}. Figure~\ref{fig:Delta2C5} gives a simple example of a 2-cut complex that does not contain a complete 2-skeleton, and is not simply connected.

As in \cite{BDJRSX2024}, it will often be helpful to examine the face lattice of a simplicial complex.   We record some facts from poset topology below. For more details see \cite{BjWachsII1997}, \cite{RPSEC11997}.

A poset $Q$ is bounded if it has a unique minimal element $\hat 0$ and a unique maximal element $\hat 1$. Its proper part is the poset $\bar{Q}=Q\setminus \{\hat 0, \hat 1\}$. The face lattice $\mathcal{L}(\Delta)$ of a simplicial complex $\Delta$ is defined to be the poset of faces ordered by inclusion, with the empty face as the minimal element, and an artificially appended maximal element.  
The order complex of $Q$ is defined to be the simplicial complex $\Delta(\bar Q)$ of chains in the proper part $\bar{Q}$ of $Q$. 
The order complex of  $\mathcal{L}(\Delta)$  is the barycentric subdivision of $\Delta$ and hence is homeomorphic to $\Delta$, and therefore  has the same homotopy type.   
Recall (see \cite{RPSEC11997}) that the reduced Euler characteristic of a simplicial complex is precisely the M\"obius number of its face lattice.  Thus, when the simplicial complex $\Delta$ is homotopy equivalent to a wedge of equidimensional spheres,  the M\"obius number of the face lattice coincides, up to sign, with the number of spheres in the wedge. In addition, when the complex $\Delta$ is not contractible, this number is also the unique nonzero Betti number.

We will need the following result of Baclawski for the work in  \Cref{sec:homology-reps}.
\begin{theorem}\label{thm:Bac-mu} \cite[Theorem 4.6]{BacEuJC1982}, 
 \cite[Lemma 3.16.4]{RPSEC11997} If $P$ is a bounded poset and $Q$ is a subposet of $P$ containing $\hat 0, \hat 1,$ then 
\[\mu(Q)-\mu(P) =\sum_{\stackrel{\hat 0<x_1<x_2<\dots<x_r<\hat 1}{r\ge 1, x_i\notin Q}} (-1)^r \mu_P(\hat 0, x_1) \mu_P(x_1,x_2) \cdots \mu_P( x_r,\hat 1).\]
where the sum runs over all nonempty chains with elements not in $Q$.  Here $\mu_P$ denotes the M\"obius function of the poset $P$.

When $Q$ is a subposet obtained from $P$ by removing an antichain $\mathcal{A},$  this simplifies to 
\begin{equation}\label{eqn:mu-deleted-antichain}
\mu(Q)-\mu(P) =\sum_{\stackrel{\hat 0<x<\hat 1} {x\in \mathcal{A}}} (-1) \mu_P(\hat 0, x)  \mu_P( x,\hat 1).
\end{equation}
\end{theorem}

We record the following  special case.

\begin{cor}\label{cor:simplicial-complex-antichain} Let $P$, $Q$ be  face lattices of simplicial complexes $\Delta$ and $\Delta_1$ respectively,  such that $\Delta_1$ is a subcomplex of $\Delta$. Assume  $P\setminus Q$ is a subset $A$ of the facets of $\Delta$. Let $\Delta$ be of dimension $d$. Then
\[ \mu(Q)- \mu(P)=(-1)^{d-1} |A| =(-1)^{d-1} \left(|\{\text{facets of } \Delta\} |
- |\{\text{facets of } \Delta_1\} |\right) . \]  
\end{cor}
\begin{proof} The hypotheses imply that the facets of $P$ have cardinality $d+1$, and thus for every $x\in A$, 
$(\hat 0, x)$ is the Boolean lattice $B_{d+1}$, while $(x,\hat 1)$ is empty, giving  
$\mu(\hat 0, x)=(-1)^{d+1}, \ \mu(x,\hat 1)=-1$.
The first equality now follows from \Cref{eqn:mu-deleted-antichain} in \Cref{thm:Bac-mu}, and the second equality is clear from the definition of $A$.  
\end{proof}

Our cut complexes $\Delta_k(G)$ often have interesting symmetries arising from the automorphism group $\mathrm{Aut}(G)$ of the graph $G$. These symmetries can be studied by examining the induced action on the homology.  In many cases the automorphism group $\mathrm{Aut}(G)$ acts simplicially on the cut complex $\Delta_k(G)$, and hence on the rational homology, when the homotopy type is a wedge of spheres.  It is well known that the face lattice $\mathcal{L}(\Delta)$ of the simplicial complex $\Delta$ captures information about the homotopy type 
via lexicographic shellings (see e.g., \cite{WachsPosetTop2007}), and the homology representation  via  the M\"obius number and the Hopf-Lefschetz trace formula  (see \cite{RPSGaP1982, RPSEC11997}).

For example, the symmetric group $\mathfrak{S}_n$ acts on the cut complex $ \Delta_k(E_n)$ of the edgeless graph $E_n$. 
 The face lattice $\mathcal{L}(\Delta_k(E_n))$ is a rank-selected subposet, denoted $B_n^{\le n-k}$,  of the Boolean lattice $B_n$ of subsets of an $n$-element set, consisting of subsets of size at most $n-k$, since the cut complex is the $(n-k-1)$-skeleton of an $(n-1)$-dimensional simplex.  It is well known \cite{BjWachsTAMS1983}, \cite{RPSGaP1982} that $\mu(B_n^{\le n-k})=(-1)^{n-k-1} \binom{n-1}{k-1}$ and that  
 $B_n^{\le n-k}$ is shellable, and therefore has the homotopy type of a wedge of spheres in the top dimension $n-k-1$.  Hence 
we have the following result for the homotopy type and equivariant homology of  $ \Delta_k(E_n)$.

\begin{prop}\label{prop:Deltak-edgeless-graph-homrepS}\cite[Proposition~3.6]{BDJRSX2024} Let $E_n$ be the edgeless graph on $n$ vertices. If $k\ge n$, the cut complex  $\Delta_k(E_n)$ is void. If $2\le k\le n-1$, the cut complex is shellable and  homotopy equivalent to a wedge of $\binom{n-1}{k-1}$ spheres in dimension $n-k-1$:
\[ \Delta_k(E_n)\simeq\bigvee_{\binom{n-1}{k-1}} \mathbb{S}^{n-k-1}.\]
 The $\mathfrak{S}_n$-representation on the unique nonvanishing homology of $\Delta_k(E_n)$ is 
the irreducible module indexed by the partition $(k, 1^{n-k})$ of $n$.
\end{prop}

\subsection{The face lattice and Betti numbers 
}\label{sec:FaceLattice-BettinNoS}
\phantom{nothing}

In this section we review a method from \cite{BDJRSX2024} for determining the Betti numbers of a cut complex, under certain favorable conditions. 

Recall \cite{RPSEC11997} that the reduced Euler characteristic of a simplicial complex $\Delta$ is the M\"obius number $\mu(\mathcal{L}(\Delta))$ of its face lattice $\mathcal{L}(\Delta)$. The idea is to describe the face lattice of the cut complex as a subposet of the truncated Boolean lattice, and then use poset topology techniques to determine the M\"obius number, using Theorem~\ref{thm:Bac-mu}.   As pointed out earlier, this is especially useful when homology is concentrated in a single degree, since in this case the M\"obius number determines the Betti number. Recall that the truncated Boolean lattice, by definition, has an artificially appended top element.

 These results 
 will be used  in \Cref{sec:Grid2021Sept15-17} to analyze the cut complexes of grid graphs. 

For a subset  $A$  of the vertex set $V(G)$ of a graph $G$, we say $A$ is a \emph{(dis)connected set}  if the induced subgraph $G[A]$ is (dis)connected. 
Let $P(n,k)$ denote the truncated Boolean lattice $B_n^{\le n-k}$, and let 
 \[\mathcal{Z}_k(G) \coloneqq \{A^c: |A|=k, A \text{ is a connected subset of }V(G)\}\subseteq P(n,k),\]
 where $A^c$ denotes the complement of $A$ in the vertex set $V(G)$. 
Thus $\mathcal{Z}_k(G)$  is the set of those $(n-k)$-element subsets of $[n]$ that are not facets of the cut complex,
 and 
the number of facets of $ \Delta_k(G)$ is $\binom{n}{k}-|\mathcal{Z}_k(G)|$.
Clearly the face lattice of $\Delta_k(G)$ is a subposet of $P(n,k)\setminus \mathcal{Z}_k(G)$.

The following theorem was proved in \cite{BDJRSX2024}.
 It will be needed in \Cref{lem:GridGraphs-k=2-4-6} and \Cref{prop:EulerChar246GridGraph}.
\begin{theorem}\cite[Theorem~6.1]{BDJRSX2024}\label{thm:truncBoolean-minus-antichain} Let $G$ be a graph with vertex set $V(G)$ of size $n$, and let $k\ge 2$.
 Then 
 the face lattice of $\Delta_k(G)$ coincides with  $P(n,k)\setminus \mathcal{Z}_k(G)$ if and only if the $(n-k-1)$-dimensional complex $\Delta_k(G)$ contains a complete $(n-k-2)$-skeleton, that is, if and only if either of the following equivalent conditions holds:
\begin{gather}\label{eqn:Condition-truncBoolean} \text{for every subset $X$ of a set $A^c\in \mathcal{Z}_k(G)$ with $|X|=n-k-1$, $X^c$ contains a disconnected set of size $k$;} \\
\label{eqn:Condition-truncBoolean-2} 
\text{if $A^c\in \mathcal{Z}_k(G)$ and  $x\notin A$, 
 there is a $y\in A$}  \text{ such that $(A\setminus\{y\})\cup\{x\}$ is disconnected. } 
\end{gather}
 If condition~\eqref{eqn:Condition-truncBoolean} holds, the reduced Euler characteristic of $\Delta_k(G)$ is given by 
\[(-1)^{n-k-1}\mu(\Delta_k(G))=\binom{n-1}{k-1}- |\mathcal{Z}_k(G)|
=|\{F: \text{$F$ is a facet of $\Delta_k(G)$}\}| -\binom{n-1}{k}.\]
Furthermore, in this case the nonzero homology of $\Delta_k(G)$ is torsion-free and occurs in at most two dimensions, $n-k-1$ and $n-k-2$.

Suppose condition~\eqref{eqn:Condition-truncBoolean} holds and $\Delta_k(G)$ is shellable.  Then
\begin{itemize}
    \item if $\mu(\Delta_k(G))=0$, i.e., if the number of facets of $\Delta_k(G)$ is $\binom{n-1}{k}$, then $\Delta_k(G)$ is contractible; 
    \item if $\mu(\Delta_k(G))\ne 0$, then $\Delta_k(G)$ is homotopy equivalent to a wedge of $\left(\binom{n-1}{k-1}- |\mathcal{Z}_k(G)|\right)$ spheres in dimension $n-k-1$.
\end{itemize}

\end{theorem}

 The next proposition will be needed in \Cref{prop:EulerChar5GridGraph}.
 \begin{prop}\label{prop:truncBoolean-minus-antichain-and-cycles} Let $G$ be a graph on $n$ vertices and let $k\ge 3$.  Let $\mathcal{Z}_k(G)$ be the set of subsets whose complements are connected of size $k$, and let $\mathcal{Y}_{k+1}(G)$  be the set of complements of $(k+1)$-cycles in the graph $G$.  

Then the face lattice of $\Delta_k(G)$ coincides with the poset 
\[P(n,k)\setminus \{\mathcal{Z}_k(G) \cup \mathcal{Y}_{k+1}(G)\}\]
if and only if the following two conditions are satisfied: 
\begin{enumerate}
    \item If $A^c\in \mathcal{Z}_k(G)$, $x\not\in A$, and $A\cup \{x\}$ is not a $(k+1)$-cycle, then there is a $y\in A$ such that $(A\setminus \{y\})\cup \{x\}$ is disconnected.
    \item If $B$ is a $(k+1)$-cycle of $G$ and $x\not\in B$, then there is a $y\in B$ such that $(B\setminus \{y\})\cup \{x\}$ is disconnected.
\end{enumerate}

The reduced Euler characteristic of $\Delta_k(G)$ is  then 
\begin{equation}\label{eqn:EulerChar-truncBoolean-minus-antichain-cycle}
\begin{gathered}
(-1)^{n-k-1}\left(\binom{n-1}{k-1} + |\mathcal{Y}_{k+1}(G)| -|\mathcal{Z}_k(G)|\right)\\
=(-1)^{n-k-1}\left(|\{F:F \text{ is a facet of }\Delta_k(G)\}|   -       \binom{n-1}{k}                                                +|\mathcal{Y}_{k+1}(G)| \right).
\end{gathered}
\end{equation}
\end{prop}

\begin{proof} The first condition in the statement guarantees the  inclusion 
\[\mathcal{L}(\Delta_k(G))\subseteq P(n,k)\setminus \{\mathcal{Z}_k(G) \cup \mathcal{Y}_{k+1}(G)\}.\]
The second condition then ensures that  no faces of the cut complex have been lost by removing all the complements of $(k+1)$-cycles.

It remains to compute the M\"obius number.
For the subposet $Q=P(n,k)\setminus \{\mathcal{Z}_k(G) \cup \mathcal{Y}_{k+1}(G)\}$ of the poset $P=P(n,k)$, using Theorem~\ref{thm:Bac-mu} in its full generality,
\begin{align*}
\mu(Q) -\mu(P) & = \sum_{A^c\in\mathcal{Z}_k(G) }(-1) \mu(\hat 0, A^c) \mu(A^c, \hat 1)
+ \sum_{B^c\in\mathcal{Y}_{k+1}(G) } (-1)\mu(\hat 0, B^c) \mu(B^c, \hat 1)\\
& \qquad + \sum_{B^c\in\mathcal{Y}_{k+1}(G) }\,\sum_{\substack{A^c\in\mathcal{Z}_k(G) \\ B^c\subset A^c}} \mu(\hat 0, B^c) \mu(B^c, A^c)  \mu(A^c, \hat 1).
\end{align*}

Since all intervals are in the Boolean lattice $B_{n}$,  the first sum is simply 
\begin{center}$(-1) |\mathcal{Z}_k(G)| (-1)^{n-k} (-1)=(-1) |\mathcal{Z}_k(G)| (-1)^{n-k-1}$. \end{center}

For the second sum, since the Boolean lattice is self-dual and $B^c\subset A^c$ if and only if 
$A\subset B,$ it is clear that the interval $(B^c, \hat 1)$ consists of $(k+1)$ subsets of size $n-k$, and hence its M\"obius number is $+k$.  Also $(\hat 0, B^c)$ is a Boolean interval of rank $(n-k-1)=|B^c|$. The second sum therefore evaluates to 
\begin{center}$(-k)\,| \mathcal{Y}_{k+1}(G)| (-1)^{n-k-1}=(-1)^{n-k}k\,| \mathcal{Y}_{k+1}(G)|  $. \end{center}

The third sum is over all chains $\hat 0\subset B^c\subset A^c\subset \hat 1$, but each $B^c$ is covered by exactly $(k+1)$ elements of type $A^c$ (obtained by removing in turn each of the $(k+1)$ vertices in the $(k+1)$-cycle $B$), so it evaluates to 
\begin{center}$(k+1)\,| \mathcal{Y}_{k+1}(G)|  (-1)^{n-k-1}$. \end{center}

Putting all this together, we find that 
$\mu(Q) -\mu(P)=(-1)^{n-k-1}\left(| \mathcal{Y}_{k+1}(G)|-|\mathcal{Z}_k(G)| \right),$ 
which is as claimed, recalling that $\mu(P)=(-1)^{n-k-1}\binom{n-1}{k-1}$.
\end{proof}

The following situation applies, for example,  to the 3-cut complex of the grid graph, and will be needed in \Cref{prop:BettiNumberDelta3GridGraph}.
\begin{prop}\label{prop:class-of-graphs2-2022-4-25} Let $G$ be a graph on $n$ vertices, let $k\ge 3$, and assume that $G$ contains no cycles of length less than or equal to $k$, but does have a cycle of length $k+1$.  Then the face lattice of $\Delta_k(G)$ coincides with the poset 
\[P(n,k)\setminus \{\mathcal{Z}_k(G) \cup \mathcal{Y}_{k+1}(G)\},\]
where $\mathcal{Z}_k(G)$ is the set of subsets whose complements are connected of size $k$, and $\mathcal{Y}_{k+1}(G)$ is the set of complements of $(k+1)$-cycles in the graph $G$.  In particular, $\Delta_k(G)$ contains a complete $(n-k-3)$-skeleton, that is, a complete codimension 2-skeleton.
The reduced Euler characteristic of $\Delta_k(G)$ is  then given by Equation~\ref{eqn:EulerChar-truncBoolean-minus-antichain-cycle} of \Cref{prop:truncBoolean-minus-antichain-and-cycles}.
\end{prop}
\begin{proof} For simplicity write $Q=P(n,k)\setminus \{\mathcal{Z}_k(G) \cup \mathcal{Y}_{k+1}(G)\}$. If $B$ is any set of size $(k+1)$, then $B^c$ is contained in a facet of the $k$-cut complex 
if and only if $B$ contains a disconnected set of size $k$. This is impossible if $B$ is a $(k+1)$-cycle, and hence the face lattice of the cut complex is a subposet of  $Q$. 

For the reverse inclusion, first let $X$ be a codimension 1 face of $\mathcal{Z}_k(G)$ that is not in $\mathcal{Y}_{k+1}(G)$. This means $X=A^c\setminus\{x\},$ i.e., $X^c=A\cup\{x\}$, where $A$ is connected of size $k$, and $X^c$ is not a $(k+1)$-cycle.  

If $X^c$ is disconnected, then $x$ is not connected by an edge to any vertex in $A$ and hence $X^c$ contains a disconnected set of size $k$, namely $(A\setminus\{u\})\cup\{x\}$ for any $u
\in A$. 

If $X^c$ is connected, and $x$ is connected by an edge to only one vertex $u$ in $A$, 
then $(A\setminus\{u\})\cup\{x\}$ is a disconnected set of size $k$ contained in $X^c$, and hence $X$ is a subset of a facet of the cut complex.
Finally $x$ cannot be connected to two  vertices $v_1\ne v_2$ of $A$, since this  would create a cycle of length at most $k+1$,
a contradiction since  there are no cycles of length $\le k$, and  $X^c$ is not a $(k+1)$-cycle.   

Now let $Y$ be a codimension 1 face of $\mathcal{Y}_{k+1}(G)$.   Then $Y=B^c\setminus\{y\}$, i.e., $Y^c=B\cup\{y\}$ where $B$ is now a $(k+1)$-cycle. 

If $y$ is connected by an edge to at most one vertex of $B$ we are done as in the argument above.  If $y$ is connected to only two vertices $u,v$ of $B$, then $\{y\}\cup B\setminus\{u,v\}$ is a disconnected set of size $k$ which is contained in $Y^c$, and hence $Y$ is contained in a facet of the cut complex. 

Let $B=\{b_1, b_2, \ldots, b_{k+1}\}$ be the $(k+1)$-cycle. The final possibility is that there are at least three vertices, say $b_1, b_i, b_j$, $1<i<j\le k+1$, such that $y$ is connected to each by an edge.  But then it is easy to see that this creates three cycles, at least one of which must have length strictly less than $k+1$. 

This establishes the reverse inclusion.  \end{proof}

\section{Disjoint Unions: $f$-{} and $h$-vectors}

We begin by recalling some facts about disjoint unions and joins of graphs from \cite{BDJRSX2024}.

\begin{df}\label{defn:Disjunion}
If $G_1$, $G_2$ are graphs, their \emph{disjoint union} is the graph $G_1+G_2$ having vertex set equal to the disjoint union of the vertex sets of $G_1$ and $G_2$, and edge set equal to the disjoint union of the edge sets of $G_1$ and $G_2$.

The \emph{join} $G_1*G_2$ of $G_1$ and $G_2$ is the graph formed from  the disjoint union by adding to the edge set of $G_1+G_2$ the set of all edges between a vertex of $V_1$ and a vertex of $V_2$. 
\end{df}
The cut complex of the disjoint union of two graphs was studied in \cite{BDJRSX2024}.  The facets of this complex are described as follows.  Let $G_1$ and $G_2$ be disjoint graphs with vertex sets $V_1$ and $V_2$.  The facets of $\Delta_k(G_1+G_2)$ are of three types: Type 1: $(|V_1|+|V_2|-k)$-sets containing some, but not all vertices of each of
$V_1$ and $V_2$; Type 2: sets of the form $V_1\cup A$, where $A$ is a $(|V_2|-k)$-subset of $V_2$
that disconnects $G_2$; and Type 3: sets of the form $B\cup V_2$, where $B$ is a $(|V_1|-k)$-subset of $V_1$ that disconnects $G_1$. This description was used to prove the following theorem.

\begin{theorem}\cite[Theorem~4.8]{BDJRSX2024}\label{thm:MargeDisjunion}
Let $k\geq2$, and $G_1$, $G_2$ graphs.   Then $\Delta_k(G_1+G_2)$ is shellable if and only if $\Delta_k(G_1)$ and $\Delta_k(G_2)$ are shellable.
\end{theorem}

Note that $\Delta_k(K_{n_1}+K_{n_2})$ is shellable, since each $\Delta_k(K_{n_i})$ is void.

\begin{prop}\label{prop:disjoint-union-facets}
Let $G_1$ and $G_2$ be disjoint graphs on $n_1$ and $n_2$ vertices, respectively.  Let $k\ge 2$.  Then $\Delta_k(G_1+G_2)=\Delta_k(G_1*G_2)\cup\Delta_k(K_{n_1}+K_{n_2})$.
\end{prop}
\begin{proof}
This follows from the description of the three types of facets of $\Delta_k(G_1+G_2)$.  Type 1 facets are those omitting at least one vertex from each of $G_1$ and $G_2$.  These are exactly the facets of $\Delta_k(K_{n_1}+K_{n_2})$.  Types~2 and~3 facets are those containing all vertices of one of the $G_i$ along with a disconnecting set of the other $G_i$.  These are exactly the facets of $\Delta_k(G_1*G_2)$. 
\end{proof}
\begin{df}[{\cite[Chapter II, Section 2]{RPSCCA1996}}]
The \emph{$f$-vector} of a $(d-1)$-dimensional simplicial complex $\Delta$ is 
the $d$-tuple $f(\Delta)=(f_0, f_1, \ldots, f_{d-1})$ where $f_i$ is the 
number of $i$-dimensional faces of $\Delta$. If $\Delta\ne \emptyset$, 
the empty set is a face of $\Delta$, and we write $f_{-1}=1$.

The \emph{$f$-polynomial} of a $(d-1)$-dimensional simplicial complex $\Delta$ 
is $f(\Delta,x)=\sum_{i=0}^d f_{i-1}x^i$.
\end{df}
We see that any description of the $f$-polynomials of the cut complexes of disjoint unions should start with the $f$-polynomials of cut complexes of joins and of the specific disjoint union $K_{n_1}+K_{n_2}$.

We will use the following simple facts about $f$-polynomials. 
\begin{prop}\label{fpoly-prop}
Let $\Delta$ and $\Gamma$ be simplicial complexes.
\begin{enumerate}
\item The $f$-polynomial of $\Delta\cup\Gamma$ is
$f(\Delta
\cup\Gamma,x)=f(\Delta,x)+f(\Gamma,x)-f(\Delta\cap\Gamma,x)$.
\item 
The $f$-polynomial of $\Delta*\Gamma$ is
$f(\Delta*\Gamma,x)=f(\Delta,x)f(\Gamma,x)$.
\item The $f$-polynomial of the  simplex $\langle V\rangle$ is
$f(\langle V\rangle,x)=(1+x)^{|V|}$.
\end{enumerate}
\end{prop}

\begin{lemma}
Suppose $G_1$ is a graph on $n_1$ vertices and $G_2$ is a graph on $n_2$ vertices. Then the $f$-polynomial of the cut complex of their join $\Delta_k(G_1*G_2)$ is
\begin{align*}
    f(\Delta_k(G_1*G_2),x)=(1+x)^{n_2}f(\Delta_k(G_1),x)+(1+x)^{n_1}f(\Delta_k(G_2),x)-f(\Delta_k(G_1),x)f(\Delta_k(G_2),x).
\end{align*}
\end{lemma}
\begin{proof}
According to \cite[Theorem~4.16]{BDJRSX2024}
$$   \Delta_k(G_1*G_2)=\Delta_k(G_1)*\langle V_2\rangle\cup \Delta_k(G_2)*\langle V_1\rangle.$$  Also, it is easy to check that $$
\Delta_k(G_1)*\Delta_k(G_2)=\Delta_k(G_1)*\langle V_2\rangle\cap \Delta_k(G_2)*\langle V_1\rangle.$$
So, by (Proposition~\ref{fpoly-prop}) 
\begin{align*}
    f(\Delta_k(G_1*G_2),x)&=f(\Delta_k(G_1)*\langle V_2\rangle,x)+f(\Delta_k(G_2)*\langle V_1\rangle,x)-f(\Delta_k(G_1)*\Delta_k(G_2),x)\\
    &=f(\Delta_k(G_1),x)f(\langle V_2\rangle,x)+f(\Delta_k(G_2),x)f(\langle V_1\rangle,x)- f(\Delta_k(G_1),x)f(\Delta_k(G_2),x)\\
    &=f(\Delta_k(G_1),x)(1+x)^{n_2}+f(\Delta_k(G_2),x)(1+x)^{n_1}-f(\Delta_k(G_1),x)f(\Delta_k(G_2),x).\qedhere
\end{align*}
\end{proof}

\begin{prop}\label{fpoly-kn-km}
The $f$-polynomial of the cut complex of the disjoint union of two cliques $\Delta_k(K_{n_1}+K_{n_2})$ is
\begin{align*}
    f(\Delta_k(K_{n_1}+K_{n_2}),x)&=\sum_{j=k}^{n_1+n_2}\left(\binom{n_1+n_2}{j}-\binom{n_1}{j}-\binom{n_2}{j}\right)x^{n_1+n_2-j}.
\end{align*}
\end{prop}

\begin{proof} A set
$\sigma\in2^{V_1\cup V_2}$ is a face of $\Delta_k(K_{n_1}+K_{n_2})$ if and only if $\sigma^c$ contains a disconnected set of size $k$, i.e., if and only if $\sigma^c\nsubseteq V_1$, $\sigma^c\nsubseteq V_2$, and $|\sigma^c|\geq k$. 
Hence for $j\ge k$ the number of $(n_1+n_2-j)$-faces equals the number of valid complements of size $j$, which is thus $\binom{n_1+n_2}{j}-\binom{n_1}{j}-\binom{n_2}{j}$. Therefore 
\begin{align*}
    f(\Delta_k(K_{n_1}+K_{n_2}),x)&=\sum_{j=k}^{n_1+n_2}\left(\binom{n_1+n_2}{j}-\binom{n_1}{j}-\binom{n_2}{j}\right)x^{n_1+n_2-j}. \qedhere
\end{align*}
\end{proof}

Now we can move to the general formula for the $f$-polynomial of the disjoint union of two graphs.

\begin{theorem}\label{fdisjunion}
Suppose $G_1$ is a graph on $n_1$ vertices and $G_2$ is a graph on $n_2$ vertices. Then the $f$-polynomial of the cut complex of their disjoint union $\Delta_k(G_1+G_2)$ is
\begin{align*}
    f(\Delta_k(G_1+G_2),x)&=x^{n_2}f(\Delta_k(G_1),x)+x^{n_1}f(\Delta_k(G_2),x)+\sum_{j=k}^{n_1+n_2}\left(\binom{n_1+n_2}{j}-\binom{n_1}{j}-\binom{n_2}{j}\right)x^{n_1+n_2-j}.
\end{align*}
\end{theorem}
\begin{proof}
From \Cref{prop:disjoint-union-facets}, we know that 
\begin{align*}
    \Delta_k(G_1+G_2)=\Delta_k(G_1*G_2)\cup\Delta_k(K_{n_1}+K_{n_2}).
\end{align*}
Consider $\Delta_k(G_1*G_2)\setminus\Delta_k(K_{n_1}+K_{n_2})$. If $\sigma\notin\Delta_k(K_{n_1}+K_{n_2})$, and $|\sigma^c|\geq k$ (the latter being true of all faces of a $k$-cut complex), then $\sigma^c\subseteq V_1$ or $\sigma^c\subseteq V_2$. In particular, $V_1\subseteq \sigma$ or $V_2\subseteq \sigma$. The faces in $\Delta_k(G_1*G_2)$ that fit this description are precisely the union of faces of $\Delta_k(G_1)$ with $V_2$ and the union of faces of $\Delta_k(G_2)$ with $V_1$. 
This gives us the formula
\begin{align}
    f(\Delta_k(G_1+G_2),x)&=x^{n_2}f(\Delta_k(G_1),x)+x^{n_1}f(\Delta_k(G_2),x)+f(\Delta(K_{n_1}+K_{n_2}),x).
\end{align}
Now we just substitute our formula from Proposition \ref{fpoly-kn-km} to arrive at the result.
\end{proof}

We turn now to a related vector, the $h$-vector of a simplicial complex.
\begin{df}\label{def:h-vec}
The \emph{$h$-vector} of a $(d-1)$-dimensional simplicial complex $\Delta$ is 
the $(d+1)$-tuple $f(\Delta)=(h_0, h_1, \ldots, h_d)$ where 
$h_j = \sum_{i=0}^j (-1)^{j-i} \binom{d-i}{j-i} f_{i-1}$.

The \emph{$h$-polynomial} of a $(d-1)$-dimensional simplicial complex $\Delta$ 
is $h(\Delta,x)=\sum_{i=0}^d h_ix^i$.

It is straightforward to check that $h(\Delta, x)=(1-x)^{d}f(\Delta,\frac{x}{1-x})$. 
\end{df}
Recall from \Cref{def:shelling} that a shelling order $F_1,F_2,\dots,F_t$  partitions the simplicial complex $\Delta$ into a disjoint union of Boolean intervals 
$\Delta=\bigcup_{i=1}^t[r(F_i), F_i],$ where the $r(F_i)$ are the restriction facets.  Then the $h$-vector $(h_0,h_1,\ldots,h_d)$ of $\Delta$ satisfies the following proposition.

\begin{prop}[{\cite[Proposition~2.3]{RPSCCA1996}}] \label{prop:hvec-from-shelling} Let $\Delta$ be a shellable $d$-dimensional simplicial complex, with shelling order $F_1,F_2,\dots,F_t$.
For $0\le i\le d$, $h_i(\Delta)$ equals the number of facets $F_j$ for which the restriction $r(F_j)$ has cardinality $i$, or equivalently, dimension $(i-1)$.
\end{prop}

\begin{prop}
Let $G_1$ and $G_2$ be disjoint graphs on $n_1$ and $n_2$ vertices,respectively.
Let $k\ge 2$.  
Then $h(\Delta_k(G_1+G_2), x)=x^{n_2}h(\Delta_k(G_1),x) + 
x^{n_1}h(\Delta_k(G_2),x) + h(\Delta_k(K_{n_1}+K_{n_2}),x)$.
\end{prop}
\begin{proof}
Note that the dimension of $\Delta_k(G_i)$ is $n_i-k-1$ and the dimension
of $\Delta_k(G_1+G_2)$ is $n_1+n_2-k-1$. From \Cref{def:h-vec} we obtain 
\begin{eqnarray*}
 \lefteqn{h(\Delta_k(G_1+G_2),x)} \\ &=&
(1-x)^{n_1+n_2-k-1}f\left(\Delta_k(G_1+G_2),\frac{x}{1-x}\right)\\
 &=&(1-x)^{n_1+n_2-k-1}\left[
\left(\frac{x}{1-x}\right)^{n_2}f\left(\Delta_k(G_1),\frac{x}{1-x}\right)+
\left(\frac{x}{1-x}\right)^{n_1}f\left(\Delta_k(G_2),\frac{x}{1-x}\right)\right.\\ 
& & \quad + \left. f\left(\Delta_k(K_{n_1}+K_{n_2}),\frac{x}{1-x}\right)\right] \\
&=& x^{n_2}(1-x)^{n_1-k-1}f\left(\Delta_k(G_1),\frac{x}{1-x}\right)+
x^{n_1}(1-x)^{n_2-k-1}f\left(\Delta_k(G_2),\frac{x}{1-x}\right)\\
 & & \quad +\ (1-x)^{n_1+n_2-k-1}
f\left(\Delta_k(K_{n_1}+K_{n_2}),\frac{x}{1-x}\right) \\
&=& x^{n_2}h(\Delta_k(G_1),x)+
x^{n_1}h(\Delta_k(G_2),x)+
h(\Delta_k(K_{n_1}+K_{n_2}),x). \qedhere
\end{eqnarray*}
\end{proof}

For the $h$-polynomial of $\Delta_k(K_{n_1}+K_{n_2})$ we use the following 
formula for the $h$-polynomial of the $j$-skeleton of a simplex. 
\begin{lemma}\cite[Lemma 3.3]{BjWachsI1996}
Let $\Delta^{j,d}$ be the $j$-skeleton of the $d$-simplex.  Then 
$$ h(\Delta^{j,d},x) = \sum_{i=0}^{j+1} \binom{d-j-1+i}{i}x^i.$$
\end{lemma}

 \begin{prop} The $h$-polynomial of $\Delta_k(K_{n_1}+K_{n_2})$ is
 $$h(\Delta_k(K_{n_1}+K_{n_2}),x)=
 \sum_{j=0}^{n_1+n_2-k}\binom{k-1+j}{k-1}x^j
 -x^{n_2}\sum_{j=0}^{n_1-k}\binom{k-1+j}{k-1}x^j
 -x^{n_1}\sum_{j=0}^{n_2-k}\binom{k-1+j}{k-1}x^j.$$
\end{prop}
\begin{proof}
Proposition~\ref{fpoly-kn-km} and
the transformation from $f$-polynomial to $h$-polynomial give 
\begin{eqnarray*}
h(\Delta_k(K_{n_1}+K_{n_2}),x)
&=&(1-x)^{n_1+n_2-k}\sum_{j=k}^{n_1+n_2}\binom{n_1+n_2}{j} \left(\frac{x}{1-x}\right)^{n_1+n_2-j}\\
& & \quad -\ (1-x)^{n_1+n_2-k}\sum_{j=k}^{n_1}\binom{n_1}{j} \left(\frac{x}{1-x}\right)^{n_1+n_2-j}\\
 & & \quad -\ (1-x)^{n_1+n_2-k}\sum_{j=k}^{n_2}\binom{n_2}{j} \left(\frac{x}{1-x}\right)^{n_1+n_2-j}.
\end{eqnarray*}
Write this as $A-B-C$.
In $A$, shifting indices with $\ell=n_1+n_2-j$, we get (since $f_{\ell-1}(\Delta^{n_1+n_2-1})=\binom{n_1+n_2}{\ell}$)
\begin{eqnarray*}
A &=&(1-x)^{n_1+n_2-k}\sum_{\ell=0}^{n_1+n_2-k}\binom{n_1+n_2}{\ell} \left(\frac{x}{1-x}\right)^{\ell}\\
&=& h(\Delta^{n_1+n_2-k-1,n_1+n_2-1},x)\\
&=& \sum_{j=0}^{n_1+n_2-k}\binom{k-1+j}{j}x^j.
\end{eqnarray*}
Similar computations show that
$$B = x^{n_2}h(\Delta^{n_1-k-1,n_1-1},x)=x^{n_2}\sum_{j=0}^{n_1-k}\binom{k-1+j}{j}x^j,$$
and 
$$C=x^{n_1}h(\Delta^{n_2-k-1,n_2-1},x)=x^{n_1}\sum_{j=0}^{n_2-k}\binom{k-1+j}{j}x^j.$$

So $ h(\Delta_k(K_{n_1}+K_{n_2}),x) = A-B-C$ gives the result.
\end{proof}
\begin{cor}\label{hdisjunion}
Let $G_1$ and $G_2$ be disjoint graphs on $n_1$ and $n_2$ vertices,respectively.
Let $k\ge 2$.  
Then 
\begin{eqnarray*}
h(\Delta_k(G_1+G_2), x)&=&x^{n_2}h(\Delta_k(G_1),x) + 
x^{n_1}h(\Delta_k(G_2),x) \\
& & \quad +
 \sum_{j=0}^{n_1+n_2-k}\binom{k-1+j}{k-1}x^j
 -x^{n_2}\sum_{j=0}^{n_1-k}\binom{k-1+j}{k-1}x^j
 -x^{n_1}\sum_{j=0}^{n_2-k}\binom{k-1+j}{k-1}x^j.
\end{eqnarray*}
In particular,
$$h(\Delta_2(G_1+G_2),x)=
(1+x+\dotsb+x^{n_1-1})(1+x+\dotsb+x^{n_2-1}) + x^{n_1}h(\Delta_2(G_2),x)+
x^{n_2}h(\Delta_2(G_1),x)
$$ and 
\begin{eqnarray*}
h(\Delta_3(G_1+G_2),x)&=& \sum_{j=0}^{n_1+n_2-3}\left[\binom{j+2}{2}-\binom{j+2-n_2}{2}-\binom{j+2-n_1}{2}\right]x^j
\\ & & \quad +\ x^{n_1}h(\Delta_3(G_2),x)+x^{n_2}h(\Delta_3(G_1),x).\qedhere
\end{eqnarray*}
\end{cor}

The formulas for $h(\Delta_2(G_1+G_2),x)$ and $h(\Delta_3(G_1+G_2),x)$, when the complexes are shellable, can also be found from their shellings.  

For shellable complexes, the $h$-vector/polynomial tells us more about the complex.  In this case,  the coefficient of $x^d$ is the number of spheres in the wedge of spheres (of dimension $d-1$) for the homotopy type of the complex.  From the above formulas we obtain
\begin{prop} Let $G_1$ be a graph with $n_1$ vertices and $G_2$ be a graph with $n_2$ vertices.
Suppose $\Delta_k(G_1)$
is shellable and is homotopy equivalent to a wedge of $w_1$ spheres (0, if it is contractible), and $\Delta_k(G_2)$ is shellable and is homotopy equivalent to a wedge of $w_2$ spheres.  Then $\Delta_k(G_1+G_2)$ is homotopy equivalent to a wedge of $w$ spheres, of dimension $(n_1+n_2-k-1)$, where $$w= w_1+w_2+
\binom{n_1+n_2-1}{k-1}-\binom{n_1-1}{k-1}-\binom{n_2-1}{k-1}.$$
In particular, under the hypotheses,
\begin{itemize}
    \item $\Delta_2(G_1+G_2)$ is homotopy equivalent to a wedge of $w_1+w_2+1$ spheres of dimension $(n_1+n_2-3)$. 
    \item $\Delta_3(G_1+G_2)$ is homotopy equivalent to a wedge of $w_1+w_2+n_1n_2-1$ spheres of dimension $(n_1+n_2-4)$.
\end{itemize}
\end{prop}

\section{Squared Paths}

\begin{df} The \emph{squared path} $P^2_n$ is the graph on the vertex set $[n]=\{1,2,\ldots,n\}$,  with edge set 
$ \{i, i+1\}_{1\le i\le n-1}\cup \{i, i+2 \}_{1\le i\le n-2} $.
\end{df}

From \cite[Theorem~4.13]{BDJRSX-TOTAL2024} the total cut complex $\Delta_k^t(P_n^2)$ of the squared path is contractible if $2\le k\le n-2$.  Since for $k=2$ this is the same as the 2-cut complex, we know that  $\Delta_2(P_n^2)$ is contractible for $n\ge 4$.  We wish to describe the cut complex of the squared path for higher $k$. We begin with a classification of the facets of this complex.

\begin{lemma}\label{lem:facets-sq-path}  Let $G=P^2_n$, $n\ge 4$, be the squared path on the vertex set $[n]$. Let $2\le k\le n-2$. Then a subset $F\subset [n]$  is a  facet of $\Delta_k(P^2_n)$ if and only if  $|F|=n-k$ and there is  some $i$, $2\le i \le n-2$, such that 
\begin{itemize}
    \item $\{i, i+1\}\subset F$, 
    \item  $i-1\notin F$,  
    \item $j\notin F$ for some $j\ge i+2$.
\end{itemize}
\end{lemma}
\begin{proof} It suffices to observe that in order for $G\setminus F$ to be disconnected, $G\setminus F$ must contain two elements that have a gap of size at least two between them.  
That is, 
$G\setminus F$ must contain two vertices $\ell< m$ such that $m-\ell\ge 3$ and $\ell+1, \ldots, m-1\in F$.
\end{proof}
\begin{ex}\label{ex:facets-sq-path-k-k+2} Let $k\ge 2$.  There are $k-1$ facets of $\Delta_k(P_{k+2}^2)$, namely, the following subsets of size 2:
$\{2,3\}, \{3,4\}, \ldots, \{k,k+1\}$.  In particular the one-dimensional complex $\Delta_k(P_{k+2}^2)$ is contractible and  shellable.
\end{ex}

\begin{theorem}\label{thm:sqpath-shellable-Betti-number}
    Let $k\ge 2$ and $n\ge k+3$. The cut complex  $\Delta_k(P^2_n)$ of the squared path $P^2_n$ is shellable.  It is homotopy equivalent to a wedge of  spheres in the top dimension $n-k-1$.
\end{theorem}
\begin{proof}
 
    Let $G=P^2_n$ be the squared path on vertex set $[n]$. 
    We use the set $\F$ of facets of $\Delta_k(P^2_n)$ determined in  Lemma~\ref{lem:facets-sq-path}. 
    Define $\F_i=\{F\in\F\mid i\notin F, i+1,i+2\in F, \text{ and } F\notin \F_j \text{ for } j<i\}$. Then $\F=\bigsqcup_{i\in[n-3]} \F_i$, as every facet must have a consecutive pair, but the groups are defined to only sort them by the first such pair. The shelling order for the facets will be such that every facet in $\F_i$ comes before every facet in $\F_{i+1}$, and if two facets are in the same $\F_i$, we order them lexicographically. We will now verify that this is a shelling order, by using the equivalent shelling criterion of Definition~\ref{def:shelling}.

    Suppose $F,H\in \F_i$ with $F$ before $H$ in lexicographic order. Let $x$ be the smallest element of $F\setminus H$ and $y$ the largest element of $H\setminus F$.  If $x<i$, $x-1\in H$, and some $z<x$ is not in $H$, then $H\cup\{x\}\setminus \{y\}\in\F_\ell$ for some $\ell < i$.  Otherwise, $H\cup\{x\}\setminus \{y\}$ is in $\F_i$ and is before $H$ in lexicographic order.  In either case $F\cap H\subset H\cup\{x\}\setminus\{y\}$, which is before $H$ in the shelling order.

    Now suppose $F\in\F_i, H\in \F_j$, with $i<j$. Suppose $H\cap \{1,2,\ldots, j\} = \{1,2,\ldots, j-1\}$. Then $i\in H$ and for $y=i$, $H\cup\{j\}\setminus\{y\}\in \F_i$.  Otherwise, let $y\in (H\setminus F)\setminus \{j+1\}$.  Then $H\cup\{j\}\setminus \{y\}$ 
    contains $j$ and $j+1$ and omits some element less than $j$ and some element greater than $j+2$.  So $H\cup\{j\}\setminus \{y\}\in \F_\ell$ for some $\ell < j$. In either case $F\cap H\subset H\cup\{j\}\setminus\{y\}$, which is before $H$ in the shelling order.
\end{proof}

A graph operation simultaneously generalizing disjoint unions and wedges of graphs  was defined in  \cite[Section 5]{BDJRSX2024} as follows. 
\begin{df}\label{def:Gen-wedge-graphs} \cite[Theorem 5.1]{BDJRSX2024} Let $G=(V,E)$ be a graph with vertex set $V$ and edge set $E$.  Suppose there is a partition  $A\cup B=V$ of $V$ such that for every edge $e\in E$, $e$ is either between two vertices of $A$ or between two vertices of $B$. Then  $G$ is the \emph{generalized wedge product} of its induced subgraphs $G[A]$ and $G[B]$.
\end{df}
It was shown in 
\cite[Theorem 5.1]{BDJRSX2024} that when $G$ is a generalized wedge product for the partition $V=A\cup B$, and $\Delta_3(A\cap B)$ is the void complex, the 3-cut complex $\Delta_3(G)$ is shellable if and only if 
$\Delta_3(G[A])$ and $\Delta_3(G[B])$ are shellable. 
We use that theorem and the following result to count the full-restriction facets of $\Delta_3(P^2_n)$ inductively.

\begin{prop}\label{prop:gen-wedge-Delta3-K3} Suppose the graph $G=(V,E)$ is a generalized wedge product  for the partition $A\cup B=V$, such that the induced subgraph $G[B]$ is $K_3$, and 
$G[A\cap B]$ is $K_2$. Suppose also that $\Delta_3(G[A])$ is shellable, with Betti number $\beta$. Then $\Delta_3(G)$ has Betti number $\beta+\gamma$ where $\gamma$ is the number of vertices $a$  such that $a$ is an isolated vertex of the induced subgraph $G[\{a\}\cup B]$.
\end{prop}
\begin{proof} Since $\Delta_3(K_3)$ and $\Delta_3(G[A\cap B])=\Delta_3(K_2)$ are both void and hence shellable (see \Cref{ex:Examples-cut-complex}),  the hypotheses of \cite[Theorem 5.1]{BDJRSX2024}  are satisfied, and therefore $\Delta_3(G)$ is shellable.  We now count the full-restriction facets (see \Cref{def:shelling}) using the facets from the shelling described in \cite[Theorem 5.1]{BDJRSX2024}. Write $S=A\cap B$. The  facets of $\Delta_3(G)$ fall into the following four categories: $\F_0, \F_1, \F_A, \F_B$, which we now recall.

We first let 
\[\F_A := \{F \in \F | F^c \subseteq A\} \quad \text{ and }\quad \F_B := \{F \in \F | F^c \subseteq B\}.\] 
In the present case, since $G[B]$ is a complete graph, 
$\F_B$ is empty, and so is the intersection $ \F_A \cap \F_B$.

We characterize the remaining facets, that is, those not in $\F_A$,  by the size of the intersection between their complement and $S$.
 Since  $|S|=2$ and $|F^c|=3$ for any facet $F$, clearly for a facet  $F$ whose complement is not contained in $A$,  we must have $|F^c\cap S|\le 1$.  Let 
\[ \F_i := \{F \in \F \setminus \F_A  \;|\; |F^c \cap S| =i\}\quad \text{ for } i= 0,1.\] 
The proof of \cite[Theorem 5.1]{BDJRSX2024} 
gives a shelling order 
with lexicographic ordering within each $\F_i$, 
for  all the facets $\F = \F_0 \sqcup \F_1 \sqcup  \F_A$. The facets in $\F_0$ are followed by those in $\F_1$, which are followed by those in $\F_A$. 
Since $\Delta_3(G[A])$ is shellable, there is already a shelling order for $\F_A$.

We now count the full-restriction facets in each of $\F_0, \F_1, \F_A$.

The facets in $\F_0$ all contain the connected set $S$, so the simplicial complex $\langle F_0\rangle$ is a cone over $S$, and therefore contractible.  The full-restriction facets of $\Delta_3(G)$ must therefore all come  from $\F_1\sqcup \F_A$. 

Label the vertex set $V$ with $\{1,2,\ldots, n\}$ and let $A=\{1,2,\ldots, n-1\}$, $B=\{n-2, n-1, n\}$ so that $S=\{n-2, n-1\}$.

The facets of $\F_0$ are all the subsets of the form \begin{equation}\label{eqn:gen-3-wedgeS}
X\cup \{n-2, n-1\}, \,  X\subseteq [n-3], \, |X|=n-5,
\end{equation}
since vertex $n$ makes no edges with any vertex in $[n-3]=\{1,2,\ldots,n-3\}$ by hypothesis.  These facets precede $\F_1$ in the shelling order.

Now let $F\in \F_1$. Then each of $F$ and $F^c$ contains exactly one of the two elements in $S$.
If $F$ is a facet in $\F_1$, then either $F^c=\{a, n-1, n\}$, $a$ such that $\{a,n-1\}$ is not an edge, or $F^c=\{b, n-2, n\}$, $b$ such that $\{b, n-2\}$ is not an edge. 
Equivalently, the facets are 
\begin{equation}\label{eqn:F1-facetS}
\begin{split}
F_{a,n-1}&=\{1, \ldots, \hat a, \ldots, n-3, n-2\},\ 
 1\le a\le n-3, \{a, n-1\}\notin E,\\ 
F_{b,n-2}&=\{1, \ldots, \hat b, \ldots, n-3, n-1\},
\ 1\le b\le n-3, \, \{b, n-2\}\notin E,
\end{split}
\end{equation}
where the hat indicates the element was omitted.
In either case the face 
$F_{a,n-1}\setminus\{n-2 \}$, $F_{b,n-2}\setminus\{n-1 \}$ cannot appear in previous facets of $\langle \F_0\rangle$, by~\eqref{eqn:gen-3-wedgeS}.

We determine which of these are full-restriction facets, that is, for which
facets all proper faces are contained in previous facets.
Suppose $F\in \F_1$, say, $F=F_{a,n-1}$. If $x\in F$, $x\ne n-2$, then 
$F\setminus \{x\} \cup \{n-1\}\in \F_0$.  So $F\setminus \{x\}$ is
contained in a previous facet.  For $x=n-2$, 
$F\setminus\{n-2\} \cap \{n-2,n-1\}=\emptyset$ and
$|F\setminus\{n-2\}| = n-4$.  So $F\setminus \{n-2\}$ is not contained
in a facet of $\F_0$ by~\eqref{eqn:gen-3-wedgeS}.

It remains to check whether $F\setminus \{n-2\}=F_{a,n-1}\setminus \{n-2\}$
 is contained in a previous
facet of $\F_1$.  If $\{a,n-2\}\in E$, then $F\setminus \{n-2\}$ is not
in a previous facet of $\F_1$, so the restriction of $F$ is contained in
$F\setminus\{n-2\}$, so $F$ is not a full-restriction facet.  On the
other hand, if $\{a,n-2\}\not\in E$, then $F_{a,n-2}$ is a facet in $\F_1$
containing $F\setminus \{n-2\}$, and 
$F_{a,n-2}$ precedes $F_{a,n-1}$ in lexicographic order.
So $F_{a,n-2}$ has restriction $F\setminus\{n-2\}$, while $F_{a,n-1}$ is a
full-restriction facet.

Thus the full-restriction facets in $\F_1$ are
 $\{F_{a, n-1}\,|\,  \{a, n-1\}\notin E, \{a, n-2\}\notin E\}$.
Hence the number $\gamma$ of full-restriction facets coming from $\F_1$ is the number of elements $a\le n-3$ such that neither $\{a, n-2\}$ nor $\{a,n-1\}$ is an edge of $G$.

If $F\in \F_A$, then $n \in F$ and $F\setminus \{n\}$ is a facet of $\Delta_3(G[A])$.  The map $F\mapsto F\setminus \{n\}$ is clearly a bijection 
between facets of $\langle\F_A\rangle$ and those of $\Delta_3(G[A])$. Since  the set of faces containing $n$ is a contractible subcomplex of $\Delta_3(G)$, removing these faces results in a complex that is homotopy equivalent to $\Delta_3(G[A])$. In particular the Betti numbers are equal, so 
 the number of full-restriction facets in $\F_A$ is the same as the number of full-restriction facets in $\Delta_3(G[A])$.

 The statement of the proposition now follows. \qedhere
 \end{proof}

\begin{cor}\label{cor:Sq-pathS} Let $n\ge 6$.  The 3-cut complex of the squared path $P_n^2$ is homotopy equivalent to a wedge of $\binom{n-4}{2}$ spheres in dimension $n-4$. There is a shelling whose full-restriction facets  are given by the complements of the sets 
\[\{ b, j-2, j\}, 1\le b\le j-5, 6\le j\le n.
\]
\end{cor}

\begin{proof} 
We apply \Cref{prop:gen-wedge-Delta3-K3} to $G=P_n^2$. Note first that in this case, the proof itself  recursively determines a shelling order for $\F_A$, since $G[A]=P_{n-1}^2$.  

We check the initial case $n=6$. 
Here $\F_A$ has facets $\{\{2,3,6\}, \{3,4,6\}\}$, and is thus contractible.  The only full-restriction facet comes from $\F_1$, and it is $\{2, 3, 5\}$, the complement of $\{1, 4, 6\}$.

The number $\gamma$ in the statement of \Cref{prop:gen-wedge-Delta3-K3} is now exactly $n-5$, since the vertices not connected to $n-2$ or to $n-1$ are $\{1,2,\ldots, n-5\}$. The full-restriction facets coming from $\F_1$ are precisely the sets $F_{b,n-2}$, and these are the complements of $\{b, n-2, n\}, b\le n-5$. Applying \Cref{prop:gen-wedge-Delta3-K3} recursively,  the full-restriction facets coming from $\F_A$, for $A=[n-1]$,  are precisely the complements of $\{b, n-3, n-1\}, b\le n-6$, in addition to the ones coming from $\Delta_3(P_{n-2}^2)$.  Continuing in this manner, the conclusion  follows.
\end{proof}

\begin{rem}\label{rem:facets-Delta3Sqpath} The above proof also shows that the  number of facets $f(3,n)$ of $\Delta_3(P_n^2)$ satisfies the recurrence 
\[f(3,n)-f(3,n-1)=\binom{n-3}{2}+ (n-4)+(n-5), \, n\ge 6.\]
\end{rem}

Table 1 contains SageMath data for the Betti numbers of the shellable cut complexes for squared paths.

\begin{table}[htbp]
\begin{center}
\scalebox{0.9}{
\begin{tabular}{|c|c|c|c|c|c|c|c|c|c|c|c|}
\hline
$k\backslash n$  & 5 &6 &7 &8 &9 &10 &11 &12 &13 &14 &15\\[2pt]\hline
$3$  &0 &1 &3 &6 &10 & 15 &21 &28 &36 & 45 & 55\\
$4$ &0 &0 & 3 &11 & 25 & 46 & 75 & 113 & 161 & 220 & 291\\
$5$ &0  &0 & 0 &6 &26 & 67 & 136 & 241 & 391 & 596 & 867\\
$6$ &0 &0 &0 & 0 & 10 & 50 & 145 & 324 & 623 & 1087 & 1771\\
$7$ &0 &0 &0 & 0 & 0 & 15 & 85 & 275 & 674 & 1403 &2627\\
$8$ & 0 &0 &0 & 0 & 0 & 0 & 21 & 133 & 476 & 1274 & 2863\\
$9$ & 0 &0 &0 & 0 & 0 & 0 & 0  & 28 & 196 & 770 & 2240\\
$10$ & 0 &0 &0 & 0 & 0 & 0 & 0  & 0 & 36 & 276 & 1182\\
$11$ & 0 &0 &0 & 0 & 0 & 0 & 0  & 0 &0 &45 & 375\\
$12$ & 0 &0 &0 & 0 & 0 & 0 & 0  & 0 &0 &0 & 55\\
\hline
\end{tabular}
}
\end{center}
\vskip .1in
\caption{\small Betti numbers $\beta(k,n)$ for the shellable complex ${\Delta_k (P_n^2)}, 3\le k\le 12, 3\le n\le 15.$}
\label{table:Betti-SquaredPath}
\end{table}
\begin{theorem}\label{thm:DeltakSqPk+3-Marija}
    $\Delta_k(P_{k+3}^2)$ is homotopy equivalent to the wedge of $\binom{k-1}{2}$ spheres $\bbS^2$, for all $k\ge 3$. 
    It has $k^2-1$ facets.
\end{theorem}

\begin{proof}
     Observe that the facets containing vertex 1 are: $\{1,3,4\}, \{1,4,5\}, \ldots, \{1,k+1,k+2\}$. Therefore vertex 1 is the apex of a cone over a path $P$ formed by the edges $\{i, i+1\}$ for $3 \le i \le k+1$. 
     
    Now we use the fact that adding a cone over a contractible subspace does not change the homotopy type, so our complex is homotopy equivalent to its subcomplex obtained by deleting vertex 1 and all faces that contain it. 
Indeed, denote by $A$ the subcomplex obtained by deleting vertex 1 and all faces that contain it (the deletion of 1), and by $B$ the subcomplex defined by all facets containing vertex 1 (the closed star of 1). Then $A$ and $B$ are closed subspaces of $\Delta_k(P_{k+3}^2)$, whose intersection is the path $P$, and whose union is $\Delta_k(P_{k+3}^2)$.
 The following topological relations hold:
$$\Delta_k(P_{k+3}^2) = A \cup B \simeq (A \cup B) / B \approx A / (A \cap B) \simeq A.$$
The two homotopy equivalences follow from Part (1) in \Cref{prop:quotient-by-contractible-homotopy} ($B$ is contractible as a cone with vertex 1, and $A \cap B = P$ is contractible as a path), while the topological equivalence follows from   \Cref{prop:TopFact2ndIsoThm}. Therefore the deletion of vertex 1 and all faces containing it does not change the homotopy type of complex $\Delta_k(P_{k+3}^2).$

    Similarly, the facets containing vertex $k+3$ are $\{2,3, k+3\}, \{3,4,k+3\}, \ldots, \{k,k+1,k+3\}$, and by the same argument, the deletion of all faces that contain vertex $k+3$ does not change the homotopy type.
    
    Hence $\Delta_k(P_{k+3}^2)$ is homotopy equivalent to its induced subcomplex $T_k$ on the vertex set $\{2, 3, \ldots, k+2\}$.  
    The face lattice $\mathcal{L}(T_k)$ of $T_k$ is a subposet of $B_{k+1}^{\le 3}$, the truncated  Boolean lattice on $k+1$ elements of sets of size at most 3.  A 3-element subset of $\{2, 3, \ldots, k+2\}$ is a facet of $T_k$ if and only if it contains at least two consecutive elements. It follows that the two-dimensional complex $T_k$  contains a full 1-skeleton, and hence its face lattice  coincides with  $B_{k+1}^{\le 3}\setminus A$, where $A$ is a subset of the facets of $B_{k+1}^{\le 3}$. In particular $A$ is an antichain, and we can apply \Cref{cor:simplicial-complex-antichain} with $P=B_{k+1}^{\le 3}$ and $Q=\mathcal{L}(T_k)$. Now $\mu(P)= (-1)^{2}\binom{k}{3}$, while the number of facets of $P=B_{k+1}^{\le 3}$ is $\binom{k+1}{3}$. A simple  calculation shows that $T_k$ has $(k-1)^2$ facets.  Since $B_{k+1}^{\le 3}$
     is the face lattice of a two-dimensional complex that is a wedge of $\binom{k}{3}$ spheres, and $T_k$ is also two-dimensional, \Cref{cor:simplicial-complex-antichain} gives us
    \[\mu(T_k)=\binom{k}{3}-\binom{k+1}{3} + (k-1)^2=\binom{k-1}{2}.\]

     By \Cref{thm:sqpath-shellable-Betti-number}, the two-dimensional complex  $T_k$ is homotopy equivalent to a wedge of spheres $\mathbb{S}^2$. It follows  that $\mu(T_k)=\binom{k-1}{2}$ is the number of spheres in the wedge.  This completes the proof. \end{proof}

Recall from \Cref{ex:facets-sq-path-k-k+2} that $\Delta_k(P^2_{k+1})$ is contractible. 
The data in Table~\ref{table:Betti-SquaredPath}  led us to make the following conjecture about the entries along the diagonals:

\begin{conj}\label{conj:Betti-numbers-SqPaths}
Let $\beta_{k}(k+r)$ be the unique nonzero  Betti number of the $(r-1)$-dimensional shellable cut complex $\Delta_k(P_{k+r}^2)$ 
for the squared path, $r\ge 3$. These numbers satisfy the following recurrence, for each FIXED $r\ge 3$ and $k\ge r+3$.
    \[{ \beta(k,k+r)=\binom{r}{1} \beta(k-1,k-1+r) -\binom{r}{2} \beta(k-2,k-2+r) +\cdots+(-1)^{r-1} \binom{r}{r} \beta(k-r,k-r+r)}.\]
We have the known values 
$\beta(3,n)=\binom{n-4}{2}$, by \Cref{cor:Sq-pathS}, $\beta(k, k+2)=0$ and $\beta(k, k+3)=\binom{k-1}{2}$, $k\ge 3$, by \Cref{thm:DeltakSqPk+3-Marija}.
\end{conj}

The first diagonal, corresponding to $r=3$, is accounted for in \Cref{thm:DeltakSqPk+3-Marija}.  The second diagonal, 
$3,11,26,50, 85, \ldots$  appears as OEIS sequence A051925.
The third diagonal, $6, 25, 67, 145, 275, 476, 770, \ldots$, 
is  OEIS A241170.		

\begin{conj}\label{conj:Betti-numbers-SqPaths-kis4-kis5}
For the unique nonzero Betti number $\beta(k,n)$  of $\Delta_k(P_n^2)$, $k=4,5$, we conjecture that 
\[\beta(4,n)=3 + 8\binom{n-7}{1} + 6\binom{n-7}{2} +\binom{n-7}{3},\,  n\ge 7,\]
\[
\beta(5,n)=6 + 20\binom{n-8}{1} + 21\binom{n-8}{2} +7\binom{n-8}{3} +\binom{n-8}{4},\, n\ge 8.\]
\end{conj}
In \cite[Conjecture 7.25]{BDJRSX2024} we conjectured the shellability of the cut complex of a  similar graph, the squared cycle on $n$ vertices, when $n\ge k+6$ and $k\ge 3$, and also gave a conjectural formula for the Betti number when $k=3$.  These conjectures were  recently proved by Chauhan, Shukla and Vinayak \cite{chauhan20243cut} for $k=3$.

\section{Grid Graphs}\label{sec:Grid2021Sept15-17}

In this section $G(m,n)$ denotes the $m$ by $n$ grid graph whose vertices are indexed by matrix-style coordinates $(i,j), 1\le i\le m, 1\le j\le n$ (so with $m$ rows and $n$ columns).   The edges of $G(m,n)$ are then $\{\{(i,j), (i+1,j)\}: 1\le i \le m-1, 1\le j\le n\}\cup\{\{(i,j), (i, j+1)\}: 1\le i\le m, 1\le j\le n-1\}.$  We assume $m\le n.$  

Here we  determine completely the homotopy type and Betti numbers of the cut complex $\Delta_k(G(m,n))$ for $k=3$ and $k=4$. We  also give a formula for the Euler characteristic when $k=6.$

\begin{prop}\label{prop:grid1dim2-by-n} For the 1-dimensional cut complex $\Delta_{2n-2}(G(2,n)), n\ge 2,$  the homotopy type is 
\[\begin{cases} \bigvee_{n-3} \mathbb{S}^1, & \mbox{if $n> 3$,}\\
                 \textrm{ a point}, & \mbox{if $n=3$,}\\
                \mathbb{S}^0, & \mbox{if $n=2$.}
                \end{cases}\]
It is shellable only if $n\ge 3.$
\end{prop}
\begin{proof} 
The graph is a rectangle of height 1, partitioned into $(n-1)$ squares. 
The cut complex has $(n-2)+2(n-1)=3n-4$ facets, each of size 2:  the $n-2$ vertical interior edges and $2(n-1)$ diagonals of squares. 
It is a $(2n-4)$-gon with $n-4$ parallel chords and four pendant edges.  It is homotopy equivalent to a wedge of $n-3$ 1-spheres. 
\end{proof}
\begin{ex}
In Figure~\ref{fig:GridGraph26}, the facets of $\Delta_{10}(2,6)$ are $\bar{1}6, 62, 25, 53, 3\bar{4}, \bar{4}\bar{2}, \bar{2}\bar{5}, \bar{5}\bar{1},$ and $2\bar{5}, \bar{2}5,$ as well as the four diagonals in the two corner squares, $16, \bar{1}\bar{6}, 34, \bar{3}\bar{4}.$ 
The four extreme corners of the grid are each in exactly one facet (edge) of the cut-complex, so these edges collapse without changing the homotopy type.
\end{ex}

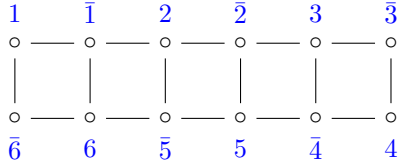
\begin{figure}[htb]
\scalebox{0.6}
\centering
\begin{tikzpicture}
\node (1) at (1,2) {$\circ$};
\node at (1,2.4) {$\textcolor{blue}{1}$};
\node (2) at (2,2) {$\circ$};
\node  at (2,2.4) {$\textcolor{blue}{\bar{1}}$};
\node (3) at (3,2) {$\circ$};
\node at  (3,2.4) {$\textcolor{blue}{2}$};
\node (4) at (4,2) {$\circ$};
\node at  (4,2.4) {$\textcolor{blue}{\bar{2}}$};
\node (5) at (5,2) {$\circ$};
\node at  (5,2.4) {$\textcolor{blue}{3}$};
\node (6) at (6,2) {$\circ$};
\node at  (6,2.4) {$\textcolor{blue}{\bar{3}}$};
\node (7) at (6,1) {$\circ$};
\node at  (6,.6) {$\textcolor{blue}{4}$};
\node (8) at (5,1) {$\circ$};
\node at  (5,.6) {$\textcolor{blue}{\bar{4}}$};
\node (9) at (4,1) {$\circ$};
\node at  (4,.6) {$\textcolor{blue}{5}$};
\node (10) at (3,1) {$\circ$};
\node at  (3,.6) {$\textcolor{blue}{\bar{5}}$};
\node (11) at (2,1) {$\circ$};
\node at  (2,.6) {$\textcolor{blue}{{6}}$};
\node (12) at (1,1) {$\circ$};
\node at  (1,.6) {$\textcolor{blue}{\bar{6}}$};
 \draw  (1) -- (2) -- (3) -- (4) -- (5) -- (6);   \draw (7) -- (8) -- (9) -- (10) -- (11) -- (12); 
\draw   (1) -- (12); \draw  (2) -- (11);   \draw (3) -- (10); 
\draw (4) -- (9); \draw (5) -- (8); \draw (6) -- (7);
\end{tikzpicture}
\caption{The Grid Graph  $G(2,6)$} \label{fig:GridGraph26}
\end{figure}

\begin{prop}\label{prop:grid1dim-m-by-n} For the 1-dimensional cut complex $\Delta_{mn-2}(G(m,n)), $  $n\ge m\ge 3,$ the homotopy type is 
\[\begin{cases} \bigvee_{3} \mathbb{S}^0, & \mbox{if $m\ge 4$,}\\
                 \mathbb{S}^0, & \mbox{if $3=m<n$,}\\
                \mathbb{S}^1, & \mbox{if $m=n=3$.}
                \end{cases}\]
It is shellable only if $m=n=3.$
\end{prop}
\begin{proof} 
The cut complex is again 1-dimensional.  There are only four facets, each of size 2: each of the four corner squares contributes one facet,  namely the diagonal which is disjoint from the corner vertex.  The facets are disjoint if $m\ge 4.$ If $3=m<n,$ they can be partitioned into two disjoint sets, each consisting of two facets  intersecting in one vertex.  If $m=n=3$, they form a 4-cycle.
\end{proof}

\begin{theorem}(\cite[Theorem~4.16]{BDJRSX-TOTAL2024}, see also \cite[Corollary 6.6]{BDJRSX2024})  \label{thm:grid-k=2MyMorseMatching2021Sept} For $n, m\ge 2,$ the $(mn-3)$-dimensional cut complex $\Delta_2 (G(m,n))$ has  homotopy type 
\[  \bigvee_{(m-1)(n-1)} \mathbb{S}^{mn-4}.\]
\end{theorem}

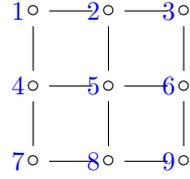
\begin{figure}[htb]
\begin{center}
\begin{tikzpicture}
\node (1) at (1,3) {$\circ$};
\node at (.8,3) {$\textcolor{blue}{1}$};
\node (2) at (2,3) {$\circ$};
\node  at (1.8,3) {$\textcolor{blue}{2}$};
\node (3) at (3,3) {$\circ$};
\node at  (2.8,3) {$\textcolor{blue}{3}$};
\node (4) at (1,2) {$\circ$};
\node at  (.8,2) {$\textcolor{blue}{4}$};
\node (5) at (2,2) {$\circ$};
\node at  (1.8,2) {$\textcolor{blue}{5}$};
\node (6) at (3,2) {$\circ$};
\node at  (2.8,2) {$\textcolor{blue}{6}$};
\node (7) at (1,1) {$\circ$};
\node at  (.8,1) {$\textcolor{blue}{7}$};
\node (8) at (2,1) {$\circ$};
\node at  (1.8,1) {$\textcolor{blue}{8}$};
\node (9) at (3,1) {$\circ$};
\node at  (2.8,1) {$\textcolor{blue}{9}$};
 \draw  (1) -- (2) -- (3); \draw  (4) -- (5) -- (6);   \draw (7) -- (8) -- (9); 
\draw   (1) -- (4) -- (7); \draw  (2) -- (5) -- (8);   \draw (3) -- (6) -- (9); 
\end{tikzpicture}
\end{center}
\caption{The Grid Graph  $G(3,3)$} 
\label{fig:GridGraph33}
\end{figure}

We need the following result from \cite{BDJRSX2024}, restated here for completeness. 

\begin{prop}\cite[Corollary~5.4]{BDJRSX2024}\label{cor:Delta3-min-forbidden-is3-conn-Results} If $G$ contains no subgraph that is 3-connected, then the 3-cut complex $\Delta_3(G)$ of graph $G$ is shellable.
\end{prop}

The following proposition is then an immediate consequence.

\begin{prop}\label{prop:MarkDelta3GridGraph}
$\Delta_3(G(m,n))$ is shellable.
\end{prop}
\begin{proof}
Let $H$ be an induced subgraph of $G=G(m,n)$. Then $H$ has a vertex of degree less than 3. Take the vertex in the top row of $H$, farthest to the left. It cannot have a neighbor above it or to the left, so it can have at most two neighbors. Thus $H$ is at best 2-connected.
 By \Cref{cor:Delta3-min-forbidden-is3-conn-Results},  $\Delta_3(G)$ is shellable.
\end{proof}

Using the results of Section~\ref{sec:FaceLattice-BettinNoS}, we are able to give a formula for the Betti number of the 3-cut complex $\Delta_3(G(m,n))$. 
Recall from Theorem~\ref{thm:truncBoolean-minus-antichain} that $\mathcal{Z}_3(G)$ is the set of complements of connected sets of size 3 in the graph $G$.  Our goal is to  compute the reduced Euler characteristic via the M\"obius number of the face lattice of $\Delta_3(G)$, following Section~\ref{sec:FaceLattice-BettinNoS}. Since the grid graph contains no 3-cycles but does contain 4-cycles, we can use Proposition~\ref{prop:class-of-graphs2-2022-4-25}  to determine the face lattice of $\Delta_3(G)$. 

Let $\tau_k(G)$ denote the number of connected sets with $k$ vertices in a graph $G$. If $G$ is a grid graph, there are no 3-cycles, so 
$\tau_3(G)$ is the number of subtrees with $3$ vertices.

\begin{lemma}\label{lem:3-trees-GridGraph}  For the grid graph $G(m,n)$, the number of subtrees on 3 vertices is 
\[\tau_3(G(m,n))=6m n-6m-6n+4.\]
\end{lemma}

\begin{proof} Note that all trees on 3 vertices are stars. The number of stars with central vertex of degree $d$ and $i$ leaves is $\binom{d}{i}$. For the grid graph $G(m,n)$, we have $d_2=4$ (the four corners), $d_3=2(m-2)+2(n-2)$ (along the perimeter) and $d_4=(m-2)(n-2)$ (the interior vertices). 
Hence 
\begin{center}$\tau_3(G(m,n))= 4\binom{2}{2}+2(m-2)\binom{3}{2}+2(n-2)\binom{3}{2}+(m-2)(n-2)\binom{4}{2}$,
 \end{center}
giving $6mn-6m-6n+4$, as claimed.
\end{proof}
\begin{prop}\label{prop:BettiNumberDelta3GridGraph} Let $G$ be the grid graph $G(m,n)$. The shellable complex $\Delta_3(G(m,n))$ has the homotopy type of a wedge of
\[\beta(\Delta_3(G(m,n)))=\binom{mn-1}{2} -5mn+5(m+n)-3, n\ge m\ge 2,\]
spheres in dimension $mn-4$.

In particular, $|\tau_3(G(2,n))|=6n-8$ and the Betti number for $G(2,n)$ is $2(n-2)^2$.

For $m=3$, $|\tau_3(G(3,n))|=2(6n-7)$ and 
the Betti number for $G(3,n)$ simplifies to the polynomial $(9n^2-29n+26)/{2}$.
\end{prop}

\begin{proof} Apply Proposition~\ref{prop:class-of-graphs2-2022-4-25} to the grid graph $G=G(m,n)$ with $k=3$. It follows that 
the face lattice of 
$\Delta_3(G)$ equals 
\[P(mn,3)\setminus (\mathcal{Z}_3(G) \cup \mathcal{Y}_4(G)),\]
where $\mathcal{Z}_3(G)$ is now the set of complements of connected trees of size 3 in the grid graph $G$,  $ \mathcal{Y}_4(G)$ is the set of complements of 4-cycles, and $P(mn,3)=B_{mn}^{\le mn-3}$ as before.   
Note that the grid graph $G(m,n)$ has $mn$ vertices, and the cut complex $\Delta_3(G)$ has dimension $(mn-4)$.  The formula of Proposition~\ref{prop:class-of-graphs2-2022-4-25} for the M\"obius function of the face lattice of $\Delta_3(G)$ now gives, for the reduced Euler characteristic, and hence in this case the unique nonzero Betti number, 
\[\binom{mn-1}{2} - |\tau_3(G(m,n))| +(m-1)(n-1),\]
since $\tau_3(G(m,n))$ is the number of trees of size 3 in the grid graph, and the number of 4-cycles is clearly $(m-1)(n-1)$.  Using Lemma~\ref{lem:3-trees-GridGraph} completes the proof.
\end{proof}

Table~\ref{table:Betti-Grid3n} below gives the Sage computation of the Betti numbers for $\Delta_k(G(3,n))$, when $m=3$ and $k\le 6$.  The homotopy type and Betti numbers for $\Delta_2, \Delta_3, \Delta_4$ and $\Delta_k$ (see below for the latter), $k\ge mn-2$, are determined in this section. $\Delta_k$ appears to be Cohen-Macaulay in all other cases, according to Sage; homology is in the top dimension. 

\begin{table}[htbp]
\begin{center}
\scalebox{0.7}{
\begin{tabular}{|c|c|c|c|c|c|l|}
\hline
$k\backslash n$ & 2  &3 & 4 & 5 &6& \\
[2pt]\hline
%
2 &2 &4 & 6& 8 & &Thm.~\ref{thm:grid-k=2MyMorseMatching2021Sept}\\
$3$ & 2 &10 &27 & 53 &  &Prop.~\ref{prop:MarkDelta3GridGraph}, Prop.~\ref{prop:BettiNumberDelta3GridGraph}\\
$4$ &0 &20 &100 & 270 & 557 & Prop.~\ref{prop:Delta4-GridGraph-2022April25}, Thm.~\ref{thm:MarkMorseMatching-GridGraph-k=4-topdim}\\
$5$ & $\emptyset$ & 25 & 221 &825   &2,137 &\\
($5$) &  & (21)  & (214)  & (815) & (2,124) & \\
$6$ & $\emptyset$ & 8 & 281 & 1,656 & &\\
\hline
\end{tabular}
}
\end{center}
\vskip .1in
\caption{\small Betti numbers for $\Delta_k (G(3,n)), 2\le k\le 6.$
Parentheses show numbers predicted by  the Euler characteristic formula in Theorem~\ref{thm:truncBoolean-minus-antichain}. 
Confirmed by Proposition~\ref{prop:EulerChar5GridGraph}.}
\label{table:Betti-Grid3n}
\end{table}
\begin{conj}\label{conj:GridGraphs} $\Delta_k(G(m,n))$ is shellable  for all $3\le k\le mn-3$.
\end{conj}
\begin{lemma}\label{lem:GridGraphs-k=2-4-6}
 Let $A$ be a connected subset of vertices of the grid graph $G=G(m,n)$ of size $k$, i.e., such that the induced subgraph $G[A]$ is connected.  If  $k\le 6$ is even, then for any vertex $x$ of $G$, $x\notin A$, there is a vertex $y\in A$ such that $A\setminus\{y\}\cup\{x\}$ is disconnected, i.e., such that $y$ is a cut vertex for the induced subgraph $G[A\cup\{x\}]$. 

In particular, the face lattice of $\Delta_k(G(m,n))$ coincides with the poset $P(mn,k)\setminus\mathcal{Z}_k(G)$, where $\mathcal{Z}_k(G)$ is the subposet whose elements are complements of connected subsets $A$ of $[mn]$ of  size $k$.  That is, $\Delta_k(G(m,n))$ contains a complete codimension 2 skeleton.
\end{lemma}
\begin{proof} Clearly we need only consider the case when $A\cup\{x\}$ is connected. Also, the claim is clear for $k=2$, so assume $k=4$ or $k=6$. 
First note that for any $k\ge 3$, the connected subsets $A$ of size $k$ correspond to  three types of induced subgraphs $G[A]$:
(1) paths, (2) trees with at least one vertex, say $y$, of degree 3,  and (3) sets  containing an induced  4-cycle of the grid graph.

The claim is clear for the first type.  For the second it suffices to note that if $A\cup\{x\}$ is connected, and $x$ is not connected to $y$ by an edge, then $y$ is a cut vertex for $G[A\cup\{x\}]$ since it is already a cut vertex for the tree $G[A]$. This is certainly true if $\{x,y\}$ is an edge of the grid graph.

Now consider the third case.  When $k=4$, $G[A]$ must be a 4-cycle, and so $x$ can be connected by an edge  to at most one vertex $y$ of $A$, and hence this vertex is a cut vertex.  When $k=6$, the same argument applies if $G[A]$ consists of two 4-cycles sharing an edge.  Otherwise, $G[A]$ consists of  a 4-cycle with two vertices $z_i, i=1,2$ connected by edges to $y_i\in A, i=1,2$ respectively, such that $\{z_1,z_2\}$ is \emph{not} an edge of the grid graph; see Figure~\ref{fig:GridGraph33} with the vertex set  $A=\{1,2,4,5,6,7\}$, where $\{z_1,z_2\}=\{6,7\}$, and $\{y_1,y_2\}=\{5,4\}$. Therefore adding the vertex $x$ introduces at most one new edge in $G[A\cup\{x\}]$, and thus at least one $y_i$ is a cut vertex of $G[A\cup\{x\}]$, since its removal isolates $z_i$.

However, $z_1$ and $z_2$ can be adjacent to the same vertex of the 4-cycle.  (For example, the $z_i$ can be vertices 6 and 8, added to $\{1,2,4,5\}$ in Figure 5.) In that case, choosing $x$ to be the vertex adjacent to both (9 in this example) adds two edges, and $A\cup \{x\}$ is the union of two 4-cycles intersecting at one vertex, which is then the needed cut vertex $y$.

 We have in fact shown that condition~\eqref{eqn:Condition-truncBoolean-2} of Theorem~\ref{thm:truncBoolean-minus-antichain} is satisfied.  The last statement now follows immediately.
\end{proof}

\begin{rem}\label{rem:False} This is false for $k\ge 8$.  When $k=8$, 
consider the vertex set $A$ such that $G[A]$ consists of three induced 4-cycles, $C_1, C_2, C_3$ such that $C_2$ shares an edge with each of $C_1, C_3$, but $C_1, C_3$ share only a vertex.  In that case we can add a vertex $x$ such $G[A\cup\{x\}]$ becomes the grid graph $G(3,3)$, (see Figure~\ref{fig:GridGraph33} with the vertex set  $A=\{1,2,\ldots,8\}$ and $x=9$), and this has no cut vertex.
\end{rem}

Theorem~\ref{thm:grid-k=2MyMorseMatching2021Sept} established the homotopy type for the case $k=2$.  The preceding lemma gives the following  result for the additional cases $k=4,6$.

\begin{prop}\label{prop:EulerChar246GridGraph} Let $G$ be the grid graph $G(m,n)$ and let $k=2,4,6$. Then the nonzero  homology of $\Delta_k(G(m,n))$ is torsion-free and concentrated in the top two dimensions, $mn-k-1$ and $mn-k-2$. The reduced Euler characteristic $\mu(\Delta_k(G(m,n)))$ satisfies
\begin{equation*}
\begin{split}
(-1)^{mn-k-1} \mu(\Delta_k(G(m,n)))
=\binom{mn-1}{k-1}-|\text{the number of connected subsets of size k in  $G(m,n)$}|\\
= |\text{the number of facets of $\Delta_k(G(m,n))$}| - \binom{mn-1}{k}.
\end{split}
\end{equation*}

\end{prop}

We have the following for  $\Delta_4(m,n)$, for which \Cref{lem:GridGraphs-k=2-4-6} shows that Theorem~\ref{thm:truncBoolean-minus-antichain} does apply.

\begin{prop}\label{prop:Delta4-GridGraph-2022April25}   The nonzero  homology of $\Delta_4(G(m,n))$ is torsion-free and concentrated in the top two dimensions, $mn-5$ and $mn-6$. The reduced Euler characteristic $\mu(\Delta_4(G(m,n)))$ satisfies
\begin{equation*}
(-1)^{mn-5} \mu(\Delta_4(G(m,n)))=\binom{mn-1}{3}-\begin{cases}
(11n-23), &\mbox{if $n> m=2$,}\\
(19mn-28(m+n) +33), &\mbox{if $n\ge m\ge 3$.} 
\end{cases}
\end{equation*}
In particular the number of facets of $\Delta_4(G(m,n))$ is 
\[\binom{mn}{4}-\begin{cases}
(11n-23), &\mbox{if $n> m=2$,}\\
(19mn-28(m+n) +33), &\mbox{if $n, m\ge 3$.} 
\end{cases}\]
\end{prop}
\begin{proof}  In view of the preceding discussion, it only remains to count the number of connected subsets of size 4 in $G(m,n)$.    

There are $(m-1)(n-1)$ 4-cycles, as in the proof of  Proposition~\ref{prop:BettiNumberDelta3GridGraph}. 

If $m\ge 3$, the number of paths of length 4 around the perimeter is $2m+2(n-2)$, since the perimeter is a $2m+2(n-2)$-cycle.

However, for $m=2$, the two  subsets of size 4 at the extreme left and right form an induced 4-cycle, so the count is $2(2) + 2(n-2)-2=2n-2$.

We count the paths NOT contained in the perimeter by their middle edge.

\begin{enumerate}
\item $m=2, n\ge 3$: There are $(n-1)$ 4-cycles, and $2n-2$ paths around the perimeter. The remaining subsets are all trees  having a distinguished central edge $\epsilon$.  

For each of the $(n-2)$ interior vertical edges  $\epsilon$, there are 4 trees with $\epsilon$ as the central edge.
This gives an additional $4(n-2)$ trees.

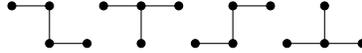
\begin{equation*}
\heha \hskip1mm \hehb \hskip1mm \hehc \hskip1mm \hehd
\end{equation*}



For each of the $2(n-3)$ interior (i.e., not touching the vertical sides) horizontal edges $\epsilon$ along the top and bottom boundaries, there are two trees  with $\epsilon$ as the central edge.
 This gives an additional $4(n-3)$ trees.

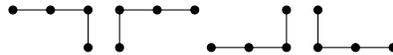
\begin{equation*}
\hevDOWN \hskip3mm \vDOWNeh \hskip1mm  \hevUP \hskip3mm  \vUPeh
\end{equation*}

The total number of subsets inducing a connected subgraph of size 4 is thus 
\[(n-1)+(2n-2)+4(n-2)+4(n-3)=11n -23.\]

\item $m\ge 3, n\ge 3$:  Now there are $(m-1)(n-1)$ 4-cycles, and $2m+2n-4$ paths around the perimeter.  Again the remaining subsets are all trees  having a distinguished central edge $\epsilon$.

\begin{enumerate}
\item[(a1)]
For each of the $2(m-2)$ interior horizontal edges $\epsilon$ touching the two vertical sides of the perimeter, there are 6 trees with central edge $\epsilon$.  
This gives $12(m-2)=12$ trees.
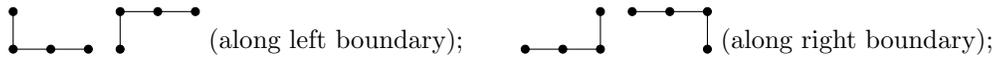
\begin{equation*}
\vUPeh \hskip3mm \vDOWNeh \text{ (along left boundary)};\qquad
\hevUP \hskip3mm  \hevDOWN  \text{ (along right boundary)};
\end{equation*}
and the following four along both left and right boundaries: 
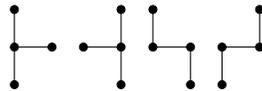
\begin{equation*}
\veva \hskip3mm \vevb \hskip3mm \vevc \hskip 3mm \vevd
\end{equation*}
%
\item[(a2)]
For each of the $2(n-2)$ interior vertical edges $\epsilon$, which touch the top and bottom horizontal sides of the perimeter, there are 6 trees having $\epsilon$ as the central edge.
This gives $12(n-2)$ trees.
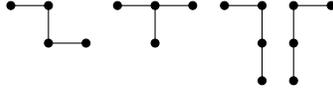
\begin{equation*}
\heha \hskip3mm \hehb \hskip3mm \hLEFTevDOWN \hskip3mm  \hRIGHTevDOWN
\end{equation*}
\item[(b1)]
For each of the $2(n-3)$ interior horizontal edges (not touching the vertical sides) $\epsilon$ along the top and bottom boundaries, there are $4(n-3)$ trees as in the case $m=2$.
\item[(b2)] Similarly for each of the $2(m-3)$ interior vertical edges (not touching the horizontal sides) $\epsilon$ along the left and right  boundaries, there are $4(m-3)$ trees.
\item[(c1)]
For each of the $(n-3)(m-2)$ interior horizontal edges $\epsilon$ not touching the perimeter, there are 9 trees with $\epsilon$ as the central edge.
This gives 
$9(n-3)(m-2)$ trees.
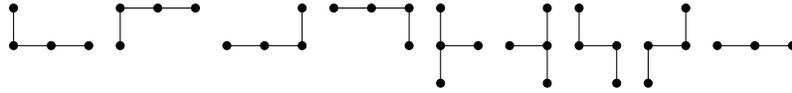
\begin{equation*}
\vUPeh \hskip3mm \vDOWNeh \hskip3mm 
\hevUP \hskip3mm  \hevDOWN  \hskip3mm 
\veva \hskip3mm \vevb \hskip3mm \vevc \hskip 3mm \vevd \hskip3mm \hhh
\end{equation*}
\item[(c2)]  Similarly, counting interior vertical edges not touching the perimeter, we have a matching item to  (c1) above, giving an additional $9(m-3)(n-2)$ trees. 
\end{enumerate}
The total number of subsets inducing a connected subgraph of size 4 is now
\begin{equation*}
\begin{split}(m-1)(n-1)+2(m+n-2)+12(n-2)+12(m-2)+4(n-3)+4(m-3)\\
+9(n-3)(m-2)+9(m-3)(n-2)
=19mn-28(m+n)+33.
\end{split}
\end{equation*}

\end{enumerate}
The  formulas in the statement now follow from  Theorem~\ref{thm:truncBoolean-minus-antichain}.
\end{proof}

We now give a Morse matching to show that   $\Delta_4(G(m,n))$ is in fact homotopy equivalent to a wedge of spheres in the top dimension. 
\begin{theorem}\label{thm:MarkMorseMatching-GridGraph-k=4-topdim}  The cut complex $\Delta_4(G(m,n))$ is homotopy equivalent to a wedge of spheres in the top dimension $mn-5$
and hence has homology concentrated in this dimension.  The 
 Betti number is given by the formula for the Euler characteristic in \Cref{prop:Delta4-GridGraph-2022April25}.
\end{theorem}

 \begin{proof} Let $v$ be the vertex in the first row and the first column, i.e., the upper left corner of the grid graph.  We start by performing an element matching with $v$. From \Cref{lem:GridGraphs-k=2-4-6}  we know $\Delta_4(G(m,n))$ is a truncated Boolean lattice with only some facets removed; thus it has a complete codimension 1 skeleton. 
\begin{figure}[htb]
    \centering
    \begin{subfigure}{0.15\textwidth}
    \centering
        \begin{tikzpicture}
            \draw (0,1) -- (1,1) -- (1,0) -- (0,0) -- (0,1);
            \draw[fill=red] (0,1) circle (4pt);
            \node at (0,1.5) {y};
            \draw[fill=black] (1,1) circle (3pt);
            \draw[fill=black] (1,0) circle (3pt);
            \draw[fill=black] (0,0) circle (3pt);
        \end{tikzpicture}
        \caption{Square}
    \end{subfigure}
    \begin{subfigure}{0.15\textwidth}
    \centering
    \begin{tikzpicture}
        \draw (0,0) -- (1,0) -- (1,1) -- (2,1);
        \draw[fill=black] (0,0) circle (3pt);
        \draw[fill=black] (1,0) circle (3pt);
        \draw[fill=red] (1,1) circle (4pt);
        \node at (1,1.5) {y};
        \draw[fill=black] (2,1) circle (3pt);
    \end{tikzpicture}
    \caption{S-Piece}
    \end{subfigure}
    \begin{subfigure}{0.15\textwidth}
    \centering
        \begin{tikzpicture}
            \draw (0,1) -- (1,1);
            \draw (1,1) -- (1,0); 
            \draw (1,0) -- (2,0);
            \draw[fill=black] (0,1) circle (3pt);
            \draw[fill=black] (1,1) circle (3pt);
            \draw[fill=red] (1,0) circle (4pt);
            \node at (0.5,0) {y};
            \draw[fill=black] (2,0) circle (3pt);
        \end{tikzpicture}
        \caption{Z-Piece}
    \end{subfigure}
    \begin{subfigure}{0.15\textwidth}
    \centering
    \begin{tikzpicture}
        \draw (1,1) -- (0,1);
        \draw (1,1) -- (1,0); 
        \draw (1,1) -- (2,1);
        \draw[fill=black] (0,1) circle (3pt);
        \draw[fill=red] (1,1) circle (4pt);
        \node at (1,1.5) {y};
        \draw[fill=black] (2,1) circle (3pt);
        \draw[fill=black] (1,0) circle (3pt);
    \end{tikzpicture}
    \caption{T-Piece}
    \end{subfigure}
    \begin{subfigure}{0.15\textwidth}
    \centering
    \begin{tikzpicture}
        \draw (1,0) -- (0,0) -- (0,1) -- (0,2);
        \draw[fill=black] (1,0) circle (3pt);
        \draw[fill=black] (0,0) circle (3pt);
        \draw[fill=red] (0,1) circle (4pt);
        \node at (0.5,1) {y};
        \draw[fill=black] (0,2) circle (3pt);
    \end{tikzpicture}
    \caption{L-Piece}
    \end{subfigure}
    \begin{subfigure}{0.15\textwidth}
    \centering
    \begin{tikzpicture}
        \draw (1,0) -- (1,1) -- (1,2) -- (1,3);
        \draw[fill=black] (1,0) circle (3pt);
        \draw[fill=red] (1,1) circle (4pt);
        \node at (1.5,1) {y};
        \draw[fill=black] (1,2) circle (3pt);
        \draw[fill=black] (1,3) circle (3pt);
    \end{tikzpicture}
    \caption{I-Piece}
    \end{subfigure}
    \caption{Possible complements of $\sigma\in X_1$ and the $y$ such that $\sigma$ is matched with $\sigma\cup\{y\}$.}
    \label{fig:Tetrominoes}
\end{figure}
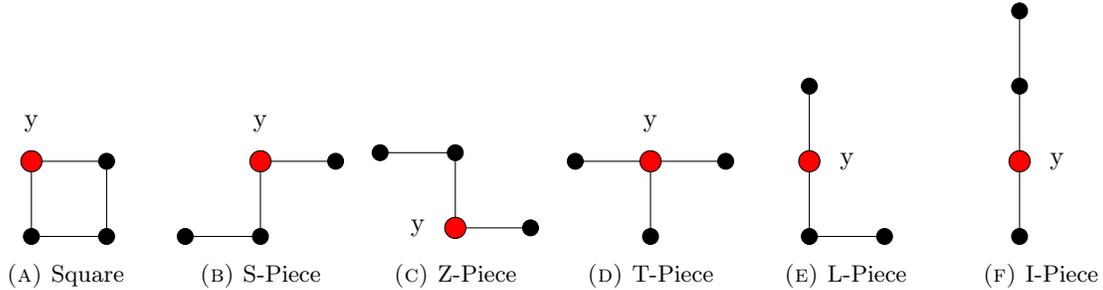
We know then that the unmatched faces must be of codimension 0 or 1.
Let $X_0$ be the facets without $v$ as an element, and let $X_1$ be the  codimension 1 faces where the addition of $v$ as an element would lead to a connected complement. We will further classify the faces $\sigma\in X_1$ by the shape of $(\sigma\cup\{v\})^c$. There are five tetrominoes up to rotation and reflection, but for technical reasons we will distinguish between what we call the S-piece and the Z-piece. The six categories are shown in \Cref{fig:Tetrominoes}.
\begin{align*}
    X_O&=\left\{ \sigma\in X_1\mid (\sigma\cup\{v\})^c \text{ is a square}\right\}\\
    X_S&=\left\{ \sigma\in X_1\mid (\sigma\cup\{v\})^c \text{ is an S-piece}\right\}\\
    X_Z&=\left\{ \sigma\in X_1\mid (\sigma\cup\{v\})^c \text{ is a Z-piece}\right\}\\
    X_T&=\left\{ \sigma\in X_1\mid (\sigma\cup\{v\})^c \text{ is a T-piece}\right\}\\
    X_L&=\left\{ \sigma\in X_1\mid (\sigma\cup\{v\})^c \text{ is an L-piece}\right\}\\
    X_I&=\left\{ \sigma\in X_1\mid (\sigma\cup\{v\})^c \text{ is an I-piece}\right\}\\
    X_1&=X_O\sqcup X_Z \sqcup X_T\sqcup X_L\sqcup X_I
\end{align*}
To match the faces in $X_1$, we must add an element that is not $v$, i.e.,  an element inside the tetronimo. The element we add to each set in each category is marked $y$ in \Cref{fig:Tetrominoes}.
Note that, as $v$ is in the top left corner of $G(m,n)$, $(\sigma\cup\{y\})^c$ is disconnected, so $\sigma\cup\{y\}\in X_0$. We claim this forms an acyclic matching. First we need to show this is a matching, which can be done by observing that none of the $(\sigma\cup\{v\cup\{y\})^c$ shown in \Cref{fig:Tetrominoes} are the same up to rotation (or reflection excluding $S$ and $Z$ pieces). So given any matched element in $X_0$, we can uniquely determine which element of $X_1$ is matched with it, even if it covers multiple elements in $X_1$.

Now we need to verify that  the matching is acyclic. The initial element matching with $v$ is known to be acyclic, so any cycle that appears must use the new matched elements,
and thus must exist between codimension 1 and 0 faces. In fact, we claim it must exist entirely between $X_0$ and $X_1$. Suppose $(\sigma_1,\ldots,\sigma_r)$ forms a cycle in the matching.  If $v\in\sigma_i$, and $\sigma_i$ is a facet, then $\sigma_i\setminus\{x\}\notin X_1$ for all $x\in G(m,n)$, as $v\notin \tau$ for all $\tau\in X_1$, so $x$ must be $v$, but as $\sigma_i$ is a facet, $\sigma_i^c$ is disconnected, so $\sigma_i\setminus\{v\}\notin X_1$ by definition of $X_1$. Now suppose $\sigma_i\notin X_1$ is a codimension 1 face; then $\sigma_i$ is matched with $\sigma_i\cup\{v\}$ which is not in $X_0$. If a cycle contains any face outside of $X_0$ or $X_1$, it must stay outside of $X_0$ and $X_1$, but an element matching is acyclic, so the cycle must be contained entirely in $X_0$ and $X_1$.

Again, suppose $(\sigma_1,\ldots,\sigma_r)$ is a cycle in the matching, $\sigma_i\in X_1$, and $\sigma_{i+1}=\sigma_i\cup\{ y\}$. Let $\sigma_{i+2}=\sigma_{i+1}\setminus\{x\}$; if $\sigma_{i+2}\in X_1$, then $(\sigma_i\cup\{ y\}\setminus\{x\})^c$ must be connected. If $\sigma_i\in X_T\cup X_L\cup X_I$, we can observe \Cref{fig:Tetrominoes} to see this is only possible if $x=y$, which is not permitted in cycles, so no cycle may contain a face in $X_T$, $X_L$, or $X_I$. Now suppose $\sigma_i\in X_S\cup X_Z$. Then we can see from \Cref{fig:S-to-L} that $\sigma_{i+2}$ could be in $X_L$, but no cycle can contain a face in $X_L$, so no cycle can contain a face in $X_S$ or $X_Z$. This leaves only $X_O$ for 
all faces of codimension 1, but if $\sigma_i\in X_O$, then $\sigma_{i+2}$ cannot be a different element of $X_O$. So the matching is in fact acyclic, and as every face in $X_1$ is matched, it has only critical cells in codimension 0.
\begin{figure}[htb]
    \centering
    \begin{tikzpicture}
        \draw (0,0) -- (1,0);
        \draw[dashed] (1,0) -- (1,1) -- (2,1);
        \draw (1,0) -- (2,0) -- (2,1);
        \draw[fill=black] (0,0) circle (3pt);
        \draw[fill=black] (1,0) circle (3pt);
        \draw[color=red, fill=white] (1,1) circle (4pt);
        \node at (1,1.5) {y};
        \draw[fill=black] (2,1) circle (3pt);
        \draw[color=blue, fill=blue] (2,0) circle (4pt);
        \node at (2.5,0) {x};
    \end{tikzpicture}
    \caption{Transforming an S-Piece into an L-Piece.}
    \label{fig:S-to-L}
\end{figure}
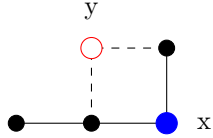
\end{proof}

For the case $k=5$, the Euler characteristic can be computed as in \Cref{prop:Delta4-GridGraph-2022April25}, by a slightly more involved analysis.
\begin{prop}\label{prop:EulerChar5GridGraph}  Let $G=G(m,n)$. The face lattice of the cut complex $\Delta_5(G)$ is isomorphic to 
\[P(mn,5)\setminus (\mathcal{Z}_5(G)\cup\mathcal{Y}_6(G)).\] 
The reduced Euler characteristic is 
\[(-1)^{n-6}\left(|\{F:F \text{ is a facet of }\Delta_5(G)\}|   -  \binom{mn-1}{5}                                                     +(m-2)(n-1)+(m-1)(n-2)\right). \]
\end{prop}

\begin{proof} We follow the argument of Lemma~\ref{lem:GridGraphs-k=2-4-6}.  The connected subsets $A$ of size $5$ are of three types:
(1) paths, (2) trees with at least one vertex, say $y$, of degree 3, or of degree 4, and (3) sets  containing an induced  4-cycle of the grid graph.
Assume $A\cup\{x\}$ induces a connected subgraph of $G$, and assume it is not a 6-cycle.  

In type (1) it is clear that $A\cup\{x\}$ has a cut vertex.
In type (2), first suppose $A$ has at least one vertex $y$ of degree 3.  If $x$ is connected by an edge to only one vertex $z\in A$, ($z$ may coincide with $y$), then $y$ is always a cut vertex. The other possibility is that $x$ is part of an induced 4-cycle, which necessarily contains $y.$ In fact $y$ and $x$ are then diagonally opposite vertices in the grid graph. In that case $y$ is still a cut vertex, since being of degree 3, it is also connected by an edge to a vertex $w$ not in the 4-cycle. 

Now suppose $A$ has a vertex $z$ of degree 4. Again either $x$ is connected by an edge to exactly one vertex $w\ne z$ of $A$, or $x$ is part of an induced 4-cycle which also contains $z$ as a  vertex diagonally opposite to $x$. In either case $z$ is a cut vertex. 

In type (3), since  $A\cup\{x\}$ is connected but not a 6-cycle, $A$ must have exactly one vertex of degree 3, and it is easy to see that this vertex is always a cut vertex for $A\cup\{x\}$. 


It remains to 
check that every codimension 1 face of a subset of $\mathcal{Y}_6(G)),$ 
 is in fact in the cut complex.  Equivalently, if $B$ is a 6-cycle and $x\notin B$ such that $B\cup\{x\}$ is connected, then $B\cup\{x\}$ has a cut vertex.  Note that the induced subgraph $G[B]$ looks like two squares sharing an edge in the grid graph.  In particular, $B$ has two vertices of degree 3,

The possibilities are that either $x$ is connected by and edge to a vertex $y$ of degree 2, or to a vertex $z$ of degree 3 of $B$.  Since $B$ has two vertices of degree 3, it is clear that either $y$ or $z$ is a cut vertex.  

The conclusion about the face lattice now follows from Proposition~\ref{prop:truncBoolean-minus-antichain-and-cycles}.   

One sees that the number of 6-cycles is the number of pairs of 4-cycles sharing an edge, and this is $(n-2)(m-1)+(m-2)(n-1)=2mn-3(m+n)+4$.  
\end{proof}

\begin{table}[htbp]
\begin{center}
\scalebox{0.7}{
\begin{tabular}{|c|c|c|c|c|c|c|l|}
\hline
$k\backslash n$ & 2  &3 & 4 & 5 &6 &7 &\\
[2pt]\hline
%
2 &1 &2 & 3& 4  & 5 &6 &Thm.~\ref{thm:grid-k=2MyMorseMatching2021Sept}\\
$3$ & 0 &2 & 8 &18 & 32 & &Prop.~\ref{prop:BettiNumberDelta3GridGraph}\\
$4$ &0 &0 &14 & 52 & 122 & 232& Prop.~\ref{prop:Delta4-GridGraph-2022April25}\\
$5$ & $\emptyset$ & $\emptyset$ & 13  &85 &270  & 636 &\\
($5$)       &              &                             & (11)  & (82)  &(266) & (631) &Thm.~\ref{thm:truncBoolean-minus-antichain}\\
$6$ & $\emptyset$ & $\emptyset$ & 1 & 71  & & &\\
\hline
\end{tabular}
}
\end{center}
\vskip .1in
\caption{\small Betti numbers for $\Delta_k (G(2,n)), 2\le k\le 6.$ 
Parentheses show Betti numbers predicted by  the Euler characteristic formula in Theorem~\ref{thm:truncBoolean-minus-antichain}.}
\label{table:Betti-Grid2n}
\end{table}

\section{Group actions on the homology}\label{sec:homology-reps}
\phantom{TODO}

It is natural to ask for the representation of a group of symmetries of our complexes on the rational  homology.  See, e.g., \cite{SolJAlg1968}, \cite{RPSGaP1982},\cite{Jer93} and the bibliography in \cite{WachsPosetTop2007}.
In this section we compute the character of this representation 
via the face lattice, using poset topology tools and the Hopf trace formula.
Similar determinations were made for the group action on the cut complex of the join of two graphs in \cite{BDJRSX2024}.  There it was shown that when the symmetric group acts on the cut complex, the irreducible representation indexed by a hook appears frequently \cite[Section 7]{BDJRSX2024}. Theorem~\ref{thm:disj-union-Kn-Km-homrep} below is one more instance of this phenomenon.

\subsection{The homology representation of $\Delta_k(K_m+K_n)$}
\label{sec:rep-disj-union}\phantom{}

Recall 
that $\Delta_k(K_m+K_n)$ is shellable, and hence its (reduced) homology vanishes in all but the top dimension $m+n-k-1$. 
The group $\mathfrak{S}_m\times \mathfrak{S}_n$  acts simplicially on the cut complex of  the disjoint union $K_m+K_n$ of complete graphs. 

We begin by describing the face lattice 
of the cut complex $\Delta_k(K_m+K_n) $. The Boolean lattice  $B_{m+n}$ can be identified with the product  $B_m\times B_n$. In particular the truncated  Boolean lattice   $B_{m+n}^{\le m+n-k}$ (which has  an artificially appended top element $\hat 1$), may be viewed as the set $\{\hat 1\}\cup \{A\sqcup B| A\subseteq [m],  B\subseteq [n], 
|A|+|B|\le m+n-k\}$. With this identification, we have the following.

\begin{lemma}\label{lem:facelatticeS-disjunion-KmKn}
The face lattice $\mathcal{L}(\Delta_k(K_m+K_n))$ of the cut complex $\Delta_k(K_m+K_n) $ is the  following subposet $L_k(m,n)$ of the truncated  Boolean lattice  $B_{m+n}^{\le m+n-k}$. 

\noindent

Let $\mathcal{L}(\Delta_k(K_m+K_n))$ be the face lattice of the cut complex
$\Delta_k(K_m+K_n) $. Then ${\mathcal{L}(\Delta_k(K_m+K_n))}$ is the
subposet $L_k(m,n)$ of the truncated  Boolean lattice  $B_{m+n}^{\le m+n-k}$
given by
$$ L_k(m,n) = B_{m+n}^{\le m+n-k}  \setminus
\mathlarger{(}\{A_1\sqcup [n]: A_1\subseteq [m], |A_1|\le m-k\} \cup
\{[m] \sqcup B_1: B_1\subseteq [n], |B_1|\le n-k\}\mathlarger{)}.$$
Thus,

if $n<k\le m$, then
$L_k(m,n) =
B_{m+n}^{\le m+n-k}  \setminus\{A_1\sqcup [n]: A_1\subseteq [m], |A_1|\le m-k\}$,

if $m<k\le n$, then
$ L_k(m,n) =
B_{m+n}^{\le m+n-k}  \setminus\{[m] \sqcup B_1: B_1\subseteq [n], |B_1|\le n-k\}$,

and if $k>\max\{m,n\}$, then
$L_k(m,n)=B_{m+n}^{\le m+n-k}$.
\end{lemma}
\begin{proof}
First note that $L_k(m,n)$ includes the top and bottom elements of $B_{m+n}^{\le m+n-k}$. It suffices to show that $\sigma$ is a face of the $k$-cut complex $\Delta_k(K_m+K_n) $ if and only if it is \emph{not} of the form $A_1\sqcup [n]$ or $[m] \sqcup B_1$ as in the statement of this lemma. By definition $\sigma$ is in $\Delta_k(K_m+K_n) $ if and only if its complement, $ [m]\sqcup [n]\setminus \sigma$, contains a disconnected set of size $k$ in $K_m + K_n$. This is true if and only if $[m]\sqcup [n]\setminus \sigma = A \sqcup B $ for nonempty subsets $A \subset [m], B\subset [n]$ such that $|A|+|B|\geq k$. Let $A_1= [m]\setminus A$ and $ B_1 = [n]\setminus B$. Equivalently, $\sigma = A_1 \sqcup B_1$ for $A_1\subsetneq [m]$, $B_1 \subsetneq [n]$ such that $|A_1| + |B_1| \leq m+n-k$. The lemma follows. \end{proof}

Recall \cite{RPSEC11997} that the reduced Euler characteristic of a simplicial complex $\Delta$ is the M\"obius number $\mu(\mathcal{L}(\Delta))$ of its face lattice $\mathcal{L}(\Delta)$.  We will compute the trace of a group element on the unique nonvanishing homology of the face lattice, which is equivalent to computing the M\"obius number of the fixed-point lattice,   by the Hopf Trace Theorem \cite[Chapter 2, Theorem~22.1]{Munkres1984}. See  \cite{Jer93} for a detailed discussion of this technique in the poset topology context.  We will need \Cref{thm:Bac-mu} for our computations. 

Let $V_{\lambda}$ denote the irreducible $\mathfrak{S}_r$-module indexed by the partition $\lambda$ of $r$.
From \cite{SolJAlg1968} it is known that 
the representation of $\mathfrak{S}_r$ on the  top homology of the truncated Boolean lattice $B_r^{\le r-k}$ is  $V_{(k,1^{r-k})}$, the module indexed by the hook $(k,1^{r-k})$.  By \Cref{prop:Deltak-edgeless-graph-homrepS}, this also coincides with the $\mathfrak{S}_r$-action on the homology of the $k$-cut complex of the edgeless graph on $r$ vertices. 
 See also \cite[Theorem 7.2]{BDJRSX2024}. In particular $V_{(1^r)}$, the  $\mathfrak{S}_r$-representation on the one-dimensional top homology of $B_r$, is the  module affording the sign representation $\sgn$ of $\mathfrak{S}_r$. We use the down arrow to indicate restriction of a module to a subgroup. 

\begin{theorem}\label{thm:disj-union-Kn-Km-homrep} Assume the cut complex $\Delta_k(K_m+K_n)$ is nonempty and  nonvoid, so that $m+n> k\ge 2$.  The representation of $\mathfrak{S}_m\times \mathfrak{S}_n$ 
on the top homology of the shellable cut complex ${\Delta_k(K_m+K_n)}$ is given by 
\[  V_{(k,1^{m+n-k})} \big\downarrow_{\mathfrak{S}_m\times \mathfrak{S}_n}^{\mathfrak{S}_{m+n}} 
-V_{(k,1^{m-k})}\otimes V_{(1^n)} - V_{(1^m)}\otimes V_{(k,1^{n-k})}. \]
If $m<k$, we interpret $V_{(k,1^{m-k})}$, and therefore $V_{(k,1^{m-k})}\otimes V_{(1^n)}$, to be the zero module, and similarly for $n<k$.

In particular, the dimension of the top homology module, i.e., the unique  nonzero Betti number of the cut complex, is 
\[\binom{m+n-1}{k-1}-\binom{m-1}{k-1}-\binom{n-1}{k-1}.\]
\end{theorem}

\begin{proof}  We will  compute the trace of an element $g\in \mathfrak{S}_m\times \mathfrak{S}_n$ on the top homology of the subposet $L_k(m,n)$. 
In this context, the Hopf trace formula says the following \cite{Jer93}. For any bounded poset $P$ and any order-preserving map $g:P\rightarrow P$, let $P^g$ denote the subposet of $P$ consisting of elements fixed pointwise by $g$. Then the trace $\tr(g)$ of $g$ on the Lefschetz module of $P$, i.e., the 
alternating sum of the rational homology modules 
of $P$, is given by the M\"obius number $\mu(P^g)$ of the fixed-point subposet $P^g$.  In particular, if homology is concentrated in a single dimension $d$, then 
\begin{equation}\label{eqn:Hopftrace-dim}
(-1)^d \mu(P^g)= \tr(g, \tilde{H}_d(P)).
\end{equation}

In the present situation, 
 the  trace of the element $g$ in $\mathfrak{S}_m\times \mathfrak{S}_n$ on the top homology is given by 
\[ (-1)^{m+n-k-1} \mu(L_k(m,n)^g).\]
Assume first that $m\ge k, n\ge k$. 

To compute the trace, we use Theorem~\ref{thm:Bac-mu}, taking $P^g$ to be the subposet  of elements of the product poset $P=(B_m\times B_n)^{\le m+n-k}$ that are fixed by $g$, and $Q^g$ to be similarly the subposet  of elements in $Q=\{\hat 0, \hat1\}\cup \{[m]\sqcup B_1, A_1\sqcup [n]: B_1\subseteq [n], |B_1|\le n-k, A_1\subseteq [m], |A_1|\le m-k   \}$ that are fixed by $g$. Note that the subposet $Q$ is invariant under the action of $\mathfrak{S}_m\times \mathfrak{S}_n$ .  Letting $(P\setminus Q)^g$ denote the subposet of $P\setminus Q$ fixed by $g$, we obtain 
\begin{equation}\label{eqn:Hopftrace}
\mu(P\setminus Q)^g-\mu(P)^g
=\sum_{\substack{c=(\hat 0<x_1<x_2<\dots<x_r<\hat 1)\\ r\ge 1\\ c \text{ a chain in $Q^g$}} }
(-1)^r \mu_P(\hat 0, x_1)^g \mu_P(x_1,x_2)^g \cdots \mu_P( x_r,\hat 1)^g.
\end{equation}

Now observe that the proper part of $Q$ is the disjoint union of 
\[X=\{A_1\sqcup [n] : A_1\subset [m], |A_1|\le m-k\} \ \text{ and } 
Y=\{[m]\sqcup B_1 : B_1\subset [n], |B_1|\le n-k\}.\]
Clearly a chain in $X\sqcup Y$ is  a chain lying either entirely in $X$ or entirely in $Y$.  Moreover, there are  $\mathfrak{S}_m\times \mathfrak{S}_n$-equivariant poset isomorphisms
\begin{equation}\label{eqn:X-Y}
X \cong B_m^{\le m-k},  \text{with $\mathfrak{S}_n$ acting trivially, and }\ \ 
Y \cong  B_n^{\le n-k}, \text{with $\mathfrak{S}_m$ acting trivially}.
\end{equation}

Denote by $X^g$ the subposet of $X$ fixed pointwise by $g$, and similarly $Y^g$.
The right-hand side of~\eqref{eqn:Hopftrace} is thus 
\begin{equation}\label{eqn:Hopftrace2}
\begin{split}
&\sum_{\substack{c=(\hat 0<x_1<x_2<\dots<x_r<\hat 1)\\ r\ge 1\\   x_i\in X^g}} 
(-1)^r \mu_P(\hat 0, x_1)^g \mu_P(x_1,x_2)^g \cdots \mu_P( x_r,\hat 1)^g \\
& +
\sum_{\substack{c'=(\hat 0<y_1<y_2<\dots<y_r<\hat 1)\\ r\ge 1\\ 
y_i\in Y^g}} 
(-1)^r \mu_P(\hat 0, y_1)^g \mu_P(y_1,y_2)^g \cdots \mu_P( y_r,\hat 1)^g.
\end{split}
\end{equation}

Now $g$ has the form $g=(\sigma_m,\sigma_n)$, for $\sigma_m\in \mathfrak{S}_m, \sigma_n\in \mathfrak{S}_n$.  Also, the bottom element of $B_m\times B_n$ is 
 $\hat 0={\hat 0}_{B_m}\sqcup {\hat 0}_{B_n}$, where 
${\hat 0}_{B_m}=\emptyset={\hat 0}_{B_n}$.   

The second expression in each claim below  is a consequence of \Cref{eqn:X-Y}.
\begin{enumerate}
\item[Claim 1:]
The first sum in~\eqref{eqn:Hopftrace2} equals 

$(-1)\, \mu_P(\hat 0, \emptyset_{B_m}\sqcup [n])^{\sigma_n} \,   \mu_P(\emptyset_{B_m}\sqcup[n],\hat 1)^{\sigma_m}
=(-1)\, \mu(B_n)^{\sigma_n} \, \mu(B_m^{\le m-k})^{\sigma_m}.$ 
\item[Claim 2:] The second sum in~\eqref{eqn:Hopftrace2} equals 

$(-1)\, \mu_P(\hat 0,  [m]\sqcup \emptyset_{B_n})^{\sigma_m} \, 
\mu_P(([m]\sqcup \emptyset_{B_n}), \hat 1)^{\sigma_n}
=(-1)\, \mu(B_m)^{\sigma_m} \,
\mu(B_n^{\le n-k})^{\sigma_n}.$
\end{enumerate}

It suffices to  show Claim 1, since Claim 2 will follow mutatis mutandis. To examine  the first sum in~\eqref{eqn:Hopftrace2}, 
we split the sum into two parts,  according to whether $x_1={\hat 0}_{B_m}\sqcup [n]$ or not.

If $x_1={\hat 0}_{B_m}\sqcup [n]$, then the subposet $(\hat 0, x_1)_P^g$ of $B_{m+n}^{\le m+n-k}$ is equivariantly poset-isomorphic to  $(B_n)^{\sigma_n}$, so 
\[\mu_P(\hat 0, x_1)^g=\mu(B_n)^{\sigma_n}.\]
If $x_1>{\hat 0}_{B_m}\sqcup [n]$, then observe that every maximal chain in the interval $(\hat 0, x_1)_P^g$ of $B_{m+n}^{\le m+n-k}$ passes through ${\hat 0}_{B_m}\sqcup [n] $, and hence 
we must have $\mu_P(\hat 0, x_1)^g=0.$ 

The first sum in~\eqref{eqn:Hopftrace2} therefore reduces to 
\begin{equation}\label{eqn:2024June28}
\begin{split}
&\sum_{\substack{(\hat 0<x_1=({\hat 0}_{B_m}\sqcup[n])<x_2<\dots<x_r<\hat 1)\\ r\ge 1\\ x_i\in X^g}}
(-1)^r \mu_P(\hat 0, x_1)^g \mu_P(x_1,x_2)^g \cdots \mu_P( x_r,\hat 1)^g\\
&=\mu(B_n)^{\sigma_n}\sum_{\substack{(x_1=({\hat 0}_{B_m}\sqcup[n])<x_2<\dots<x_r<\hat 1)\\ r\ge 1\\ x_i\in X^g}}
(-1)^r  \mu_P(x_1,x_2)^g \cdots \mu_P( x_r,\hat 1)^g\\
&=\mu(B_n)^{\sigma_n}  \mu_P({\hat 0}_{B_m}\sqcup[n],\hat 1)^{\sigma_m} (-1)\\
&=(-1)\mu(B_n)^{\sigma_n} \mu(B_m^{\le m-k})^{\sigma_m},
\end{split}
\end{equation}
thereby establishing Claim 1.  Claim 2 follows similarly.

We have shown that, with $g=(\sigma_m, \sigma_n)$,~\eqref{eqn:Hopftrace} becomes 
\[\mu(P\setminus Q)^g-\mu(P)^g= (-1)\mu(B_n)^{\sigma_n} \mu(B_m^{\le m-k})^{\sigma_m} +(-1)\mu(B_m)^{\sigma_m}  \mu(B_n^{\le n-k})^{\sigma_n}.\]
It remains to observe that, using~\eqref{eqn:Hopftrace-dim}, this can be rewritten as 
\begin{equation*}
\begin{split}
&(-1)^{m+n-k-1} \tr(g, \tilde{H}(\mathcal{L}(\Delta_k(K_m+K_n))) 
-(-1)^{m+n-k-1} \tr(g, \tilde{H}(B_{m+n}^{\le m+n-k})\\
&= (-1) \left((-1)^{n-2} \tr(\sigma_n, \tilde{H}(B_n)) (-1)^{m-k-1}\tr(\sigma_m, \tilde{H}(B_{m}^{\le m-k})\right)\\
&+(-1) \left((-1)^{m-2} \tr(\sigma_m, \tilde{H}(B_m)) (-1)^{n-k-1}\tr(\sigma_n, \tilde{H}(B_{n}^{\le n-k})\right).
\end{split}
\end{equation*} 
since $\Delta_k(K_m+K_n)$ and the order complex of $B_{m+n}^{\le m+n-k}$ both have dimension $(m+n-k-1)$,  the order complex of $B_n$ has dimension $n-2$, while that of $B_n^{\le n-k}$ has dimension $n-k-1$. 

By a result of Solomon \cite{SolJAlg1968}, see also \cite[Proposition 3.6 and Theorem 7.2]{BDJRSX2024}, the $\mathfrak{S}_n$-homology module of $B_n^{\le n-k}$ is $V_{(k, 1^{n-k})}$, and in particular $\tr(\sigma_m, \tilde{H}(B_m))=\sgn(\sigma_m)$. This gives  
the $\mathfrak{S}_m\times \mathfrak{S}_n $-isomorphism 
\[\tilde{H}(\mathcal{L}(\Delta_k(K_m+K_n)))\cong 
V_{(k,1^{m+n-k})} \big\downarrow_{\mathfrak{S}_m\times \mathfrak{S}_n} 
-V_{(k,1^{m-k})}\otimes V_{(1^n)} - V_{(1^m)}\otimes V_{(k,1^{n-k})},\]
as claimed.

Now assume one of $m,n$ is less than $k$, say $n<k$.  (Recall that $\Delta_k(K_m+K_m)$ is nonvoid.) In that case the set $Y$ is empty, and we only need  Claim 1, giving the $\mathfrak{S}_m\times \mathfrak{S}_n $-isomorphism 
\[\tilde{H}(\mathcal{L}(\Delta_k(K_m+K_n)))\cong 
V_{(k,1^{m+n-k})} \big\downarrow_{\mathfrak{S}_m\times \mathfrak{S}_n}
-V_{(k,1^{m-k})}\otimes V_{(1^n)}.
\]

Finally if $m<k$, $n<k$ but $m+n>k$, and $\Delta_k(K_m+K_m)$ is nonvoid, then one sees that the face lattice $\mathcal{L}(\Delta_k(K_m+K_n)))$ of the cut complex coincides with the truncated Boolean lattice $B_{m+n}^{\le m+n-k}$, and the homology representation is $V_{(k,1^{m+n-k})}$.
\end{proof}

\subsection{Paths and Cycles}\label{sec:Traces-path-cycle}
The goal of this subsection is to use the group equivariant version of \Cref{thm:truncBoolean-minus-antichain} to compute the character of the group representation on the homology of $\Delta_k(G)$, when 
 $G$ is the graph $P_n$ or the cycle graph $C_n$.   
 It was shown in \cite[Propositions 7.12, 7.13]{BDJRSX2024} that $\Delta_k(G)$ satisfies the conditions of \Cref{thm:truncBoolean-minus-antichain}, that it is homotopy equivalent to a wedge of equidimensional spheres (and shellable for $k\ge 3$), and hence more precisely that 
\begin{enumerate}
    \item $\Delta_k(P_n) \sim $ 
    $\begin{cases} \text{ a wedge of  }\binom{n-1}{k-1} -(n-k+1) \text{ spheres }  \mathbb{S}^{n-k-1}, & \mbox{if $k\ge 3$,} \\
    \text{ contractible}, & \mbox{if $k=2$.}
     \end{cases}$
     \item $\Delta_k(C_n)\sim$ 
    $\begin{cases} \text{ a wedge of  }\binom{n-1}{k-1} -n \text{ spheres } \mathbb{S}^{n-k-1} , & \mbox{if $k\ge 3$,} \\
    \text{ one sphere } \mathbb{S}^{n-k-2}, & \mbox{if $k=2$.}
     \end{cases}$
\end{enumerate}

 The automorphism group of $P_n$ is generated by an involution $\rho\in \mathfrak{S}_n$.  In this section we will compute the trace of $\rho$ on the unique nonvanishing homology of the cut complex $\Delta_k(P_n)$, i.e., the value of the character on $\rho$, thereby determining the action completely.  A similar but more involved analysis is required for the cut complex $\Delta_k(C_n)$ of the cycle graph $C_n$, whose group of symmetries is the dihedral group $I_2(n)$.

 Recall from \cite[Propositions 7.12, 7.13]{BDJRSX2024} that if the graph $G$ is either the path $P_n$ or the cycle $C_n$, then  the face lattice of the  cut complex $\Delta_k(G)$ coincides with $B_n^{\le n-k}\setminus \mathcal{Z}_k(G)$, where $\mathcal{Z}_k(G)$ is as in \Cref{thm:truncBoolean-minus-antichain}, namely 
 \[\mathcal{Z}_k(G) \coloneqq \{A^c: |A|=k, A \text{ is a connected subset of }V(G)\}\subseteq P(n,k),\]
 where $A^c$ denotes the complement of $A$ in the vertex set $V(G)$.

As in Section~\ref{sec:rep-disj-union} we compute the trace of a group element $g$ on the unique nonvanishing homology by computing the M\"obius numbers of the fixed-point subposets under $g$.  If $Q$ is any poset, write $Q^g$ for the subposet $Q$ fixed by the automorphism $g$. 

We apply \Cref{eqn:mu-deleted-antichain} of \Cref{sec:background} to the respective fixed-point subposets, using the facts that $\dim \Delta(P(n,k))=n-k-1$, and that $(x,\hat 1)$ is just the Boolean lattice $B_1$ with M\"obius number $-1$.

From \Cref{eqn:Hopftrace-dim}, 
we therefore have, for  $G=P_n \text{ or } C_n$,  
\begin{equation}\label{eqn:trace-kNOT2} \mathrm{tr}(g, \tilde{H}_{n-k-1}(\Delta_k(G)) -
\mathrm{tr}(g, \tilde{H}_{n-k-1}(P(n,k)))=(-1)^{n-k-1}\sum_{\stackrel {x\in \mathcal{Z}_k(G)}{g(x)=x}} \mu_{P(n,k)}(\hat 0, x)^g , \text{ if $k\ge 3$},
\end{equation}
and, recalling that for $k=2$, homology is concentrated in dimension one below the top for $\Delta_2(C_n)$, 
\begin{equation}\label{eqn:trace-kIS2} -\mathrm{tr}(g, \tilde{H}_{n-3}(\Delta_2(C_n)) -
\mathrm{tr}(g, \tilde{H}_{n-3}(P(n,2)))=(-1)^{n-3}\sum_{\stackrel {x\in \mathcal{Z}_2(C_n)}{g(x)=x}} \mu_{P(n,k)}(\hat 0, x)^g , \text{ if $k=2$}.
\end{equation}

Equivalent formulas  in terms of virtual modules for the group action  on the homology  (for any graph $G$) are as follows.  See also \cite[Proof of Theorem 7.15]{BDJRSX2024}. We omit the homology degrees for ease of notation.
\begin{equation}\label{eqn:equiv-hom-kNOT2}  \tilde{H}(\Delta_k(G)) -
\tilde{H}(P(n,k))=\bigoplus_{\stackrel {x\in \mathcal{Z}_k(G)}{g(x)=x}} \tilde{H}(\hat 0, x)\otimes \tilde{H}(x, \hat 1) , \text{ if $k\ge 3$},
\end{equation}
\begin{equation}\label{eqn:equiv-hom-kIS2}  \tilde{H}(\Delta_2(C_n))=\bigoplus_{\stackrel {x\in \mathcal{Z}_k(C_n)}{g(x)=x}} \tilde{H}(\hat 0, x)\otimes \tilde{H}(x, \hat 1) 
-
\tilde{H}(P(n,2)), \text{ if $k=2$}.
\end{equation}
The lemmas that follow will compute the M\"obius number of the fixed-point subposet $(\hat 0,x)^\rho$ of $P(n,k)$, for $x\in \mathcal{Z}_k(G)$. Recall that the $\mathfrak{S}_n$-representation on the homology of the truncated Boolean lattice $P(n,k)=B_n^{\le n-k}$ is the irreducible  module $V_{(k,1^{n-k})}$ indexed by the partition $(k, 1^{n-k})$ of $n$. 
\begin{lemma}\label{lem:fixed-points-Zk(Pn)} 

Label the vertices of $P_n$ by $\{1,\ldots,n\}$ in order.  Let $\rho$ be the automorphism of $P_n$ defined by $\rho(i)=n+1-i$.  
The number of facets in $\mathcal{Z}_k(P_n)$ fixed by $\rho$ is 
0 if $n-k$ is odd and 
1 if $n-k$ is even.  The M\"obius number of the fixed-point subposet $(\hat 0, x_0)$ in the latter case for the facet $x_0$ fixed by $\rho$ is $(-1)^{\frac{n-k}{2}}$.

Hence the trace of $\rho$ on the unique nonvanishing homology of the cut complex $\Delta_k(P_n)$, $k\ge 3$, is given by 
\[\mathrm{tr}(\rho, \tilde{H}(\Delta_k(P_n)) 
= \chi^{(k, 1^{n-k})}(\rho)
+(-1)^{n-k-1}
\begin{cases} 0, & \textrm{if $n-k$ is odd},\\
(-1)^{\frac{n-k}{2}}, & \textrm{if $n-k$ is even.}
\end{cases}
\]
\end{lemma}
\begin{proof} By definition,  the set 
 $\mathcal{Z}_k(P_n)$ consists of all intervals $[i, i+k-1]$, 
$1\le i\le n-k+1$. 

When $n-k$ is odd, there are no facets in $\mathcal{Z}_k(P_n)$ fixed by $\rho$.

When $n-k$ is even, the only facet in $\mathcal{Z}_k(P_n)$ fixed by $\rho$ is 
$x_0=[1, \frac{n-k}{2}]\cup [ \frac{n+k}{2} ,n]$. 
In this case the fixed-point subposet $(\hat 0, x_0)^\rho$ is easily seen to be isomorphic to the Boolean lattice on $\frac{n-k}{2}$ elements, since $x_0$ is the union of the elements in the set of 2-cycles  $\{(1,n), (2, n-1), \ldots, (\frac{n-k}{2}, \frac{n+k}{2}) \}$, and thus its M\"obius number is $(-1)^{\frac{n-k}{2} }$. 
The trace computation follows directly from \Cref{eqn:trace-kNOT2}. \end{proof}

The automorphism group $\mathrm{Aut}(P_n)$ is the cyclic group of order 2 generated by the unique automorphism $\rho$  of the path graph $P_n$. 
When $n-k$ is even,  let $\epsilon_{n,k}$ denote the one-dimensional representation of $\mathrm{Aut}(P_n))=\langle\rho\rangle$ defined by $\rho\mapsto (-1)^{\frac{n-k}{2}}$.  Let $\mathrm{Reg}_{\mathrm{Aut}(P_n) }$ denote the regular representation of $\mathrm{Aut}(P_n)$.  

\begin{prop}\label{prop:Z2-action-path} As an $\mathrm{Aut}(P_n)$-module,  we have that  
when $k\ge 3$:
\[\tilde{H}_{n-k-1}(\Delta_k(P_n)) =
\begin{cases}
 V_{(k, 1^{n-k})}\big\downarrow^{\mathfrak{S}_n}_{\mathrm{Aut}(P_n)} - \frac{n-k+1}{2} \,\mathrm{Reg}_{\mathrm{Aut}(P_n)}, & \textrm{if $n-k$ is odd},\\
V_{(k, 1^{n-k})}\big\downarrow^{\mathfrak{S}_n}_{\mathrm{Aut}(P_n)}  - \frac{n-k}{2}\, \mathrm{Reg}_{\mathrm{Aut}(P_n)} -\epsilon_{n,k}, & \textrm{if $n-k$ is even.}
\end{cases}
\]
\end{prop}

\begin{proof}  We apply the calculation of \Cref{lem:fixed-points-Zk(Pn)} to \Cref{eqn:equiv-hom-kNOT2}. The key observation is that when $n-k$ is even, $\rho$ fixes one point in $\mathcal{Z}_k(G)$, with the action on $\tilde{H}(\hat 0, x_0)$ given by $\epsilon_{n,k}$. Each of the remaining orbits has  size 2, and is of the form $\{S, \rho(S)\}$. The character computation  shows that this  gives  one copy of the regular representation of $\langle \rho\rangle$ for each such pair,  in the sum in the right-hand side of \Cref{eqn:equiv-hom-kNOT2}. When $n-k$ is odd, there are only orbits of size 2, and the same analysis applies.

The decomposition now follows from  \Cref{eqn:equiv-hom-kNOT2}. \end{proof}

Now we turn to the cycle graph $C_n$.  The action of the cyclic group $\mathfrak{C}_n$ on the unique nonvanishing homology of $\Delta_k(C_n)$ was described in \cite[Theorem~7.15]{BDJRSX2024}.  The full automorphism group of the graph  $C_n$ is the dihedral group $I_2(n)$ of order $2n$, generated by  $\mathfrak{C}_n$ and an involution $\rho$. Computing the trace of $\rho$   determines the character value on the conjugacy class of all the reflections when $n$ is odd.   When $n$ is even, the reflections form two conjugacy classes, one containing $\rho$ and the other containing the reflection  $\tau=\rho\sigma$, where $\sigma=(1,2,\ldots,n)$ is the long cycle in $\mathfrak{S}_n$, corresponding to a rotation by $\frac{2\pi}{n}$.   Thus  when $n$ is even, we need to compute the traces of both $\rho$ and $\tau=\rho\sigma$. See \cite[Sec. 5.3]{Serre1977} for details about $I_2(n)$. 
The group $I_2(n)$ acts on the face lattice of $\Delta_k(C_n)$. For any face $F$, we write $\mathrm{stab}(F)$ for the stabilizer subgroup $\{g\in I_2(n): g\cdot F= F\}$, and  $\langle g \rangle$ for the cyclic subgroup generated by the group element $g$.

\begin{lemma}\label{lem:fixed-points-Zk(Cn)} 

Label the vertices of $C_n$ by $\{1,\ldots,n\}$ in order.  
\begin{enumerate}
\item Suppose $n$ is even. 

Let $\rho$ be  the fixed-point-free involution $(1,n)(2, n-1)\cdots (\frac{n}{2}, \frac{n}{2}+1)$. 
The  number of facets in $\mathcal{Z}_k(C_n)$ that are fixed by $\rho$ is 0 if $k$ is odd and 2 if $k$ is  even. 
 For each  facet $A^c$ in $\mathcal{Z}_k(C_n)$ fixed by $\rho$, we have $\mathrm{stab}(A^c)=\langle \rho\rangle$ and 
 \begin{center}
$\mu(\hat 0, A^c)^\rho=(-1)^{\frac{n-k}{2}}$.
 \end{center}
 Let $\tau$ be  the  involution 
 $(n)(\frac{n}{2})(1, n-1)(2, n-2)
 \cdots(\frac{n}{2}-1, \frac{n}{2}+1)$, having two fixed points.
  The number of facets in $\mathcal{Z}_k(C_n)$ that are fixed by $\tau$ is 2 if $k$ is odd and 0 if $k$ is even.
 For each  facet $B^c$ in $\mathcal{Z}_k(C_n)$ fixed by $\tau$, we have $\mathrm{stab}(B^c)=\langle \tau\rangle$ and 
 \begin{center}
$\mu(\hat 0, B^c)^\tau=(-1)^{\frac{n+1-k}{2}}$.
\end{center}
\item Suppose $n$ is odd.

Let $\rho$  be the involution   
$(1,n-1)(2, n-2)\cdots (\frac{n-1}{2}, \frac{n+1}{2}) (n)$ with one fixed point $n$. 
Then exactly one facet $A^c$ in $\mathcal{Z}_k(C_n)$ is fixed by $\rho$.  We have $\mathrm{stab}(A^c)=\langle \rho\rangle$ and 
\[\mu(\hat 0, A^c)^\rho=
\begin{cases} (-1)^{\frac{n-k}{2}}, & \mbox{if $k$ is odd},\\
(-1)^{\frac{n+1-k}{2}}, & \mbox{if $k$ is even}.
\end{cases}\]
\end{enumerate}
\end{lemma}
 \begin{proof}
The facets in $\mathcal{Z}_k(C_n)$ are the complements of the sets of consecutive integers 
$A=\{j, j+1,\ldots, j+k-1\}$, $1\le j\le n$ (taken modulo $n$). 
The statements follow by an analysis similar to the proof of \Cref{lem:fixed-points-Zk(Pn)}, the key observation being that a subset is fixed by a permutation $\sigma\in\mathfrak{S}_n$ if and only if it is a union of elements in some subset of the disjoint cycles of $\sigma$.
\end{proof}

Let $\sigma$ be the long cycle in the cyclic subgroup $\mathfrak{C}_n$ of $\mathfrak{S}_n$. 
Compare the next result to \cite[Theorem 7.15]{BDJRSX2024}.
 Down and up arrows denote restriction and induction respectively.
\begin{theorem}\label{prop:dihedral-action-cycle} 
For the action of the dihedral group $I_2(n)$,  the unique nonvanishing homology  module of the cut complex $\Delta_k(C_n)$, $2\le k\le n-2$ can be described by the following formulas.  

\begin{enumerate}
\item Let $k\ge 3$ and suppose $n$ is even.  Let $\rho$, $\tau=\rho\sigma$ be the reflections in Part (1) of \Cref{lem:fixed-points-Zk(Cn)}.   Let $\hat\rho_{n,k}, \hat\tau_{n,k}$ denoted the one-dimensional  representation of $\langle \rho\rangle, \, \langle \tau\rangle$, respectively, defined by $\rho\mapsto (-1)^{\frac{n-k}{2}}$ if $k$ is even, and by  $\tau\mapsto (-1)^{\frac{n-k+1}{2}}$ if $k$ is odd.  Then 
\begin{equation*}
    \tilde{H}_{n-k-1}(\Delta_k(C_n))
    =\begin{cases}
    V_{(k, 1^{n-k})}\big\downarrow_{I_2(n)}^{\mathfrak{S}_n}
    - \hat\rho_{n,k}\big\uparrow_{\langle \rho\rangle}^{I_2(n)}, & \text{ $k$ even},\\
     V_{(k, 1^{n-k})}\big\downarrow_{I_2(n)}^{\mathfrak{S}_n}
    - \hat\tau_{n,k}\big\uparrow_{\langle \tau\rangle}^{I_2(n)}, & \text{ $k$ odd}.
    \end{cases}
\end{equation*}

\item Let $k\ge 3$ and suppose $n$ is odd.  Let $\rho$ be the unique reflection of $C_n$ fixing $n$ as in Part (2) of \Cref{lem:fixed-points-Zk(Cn)}, and let $\hat\rho_{n,k}$ be the one-dimensional representation of $\langle \rho\rangle$ defined by $\rho\mapsto (-1)^{\lfloor \frac{n-k+1}{2}\rfloor}$.  Then  
\begin{center}
    $\tilde{H}_{n-k-1}(\Delta_k(C_n))
    =
    V_{(k, 1^{n-k})}\big\downarrow_{I_2(n)}^{\mathfrak{S}_n}
    - \hat\rho_{n,k}\big\uparrow_{\langle \rho\rangle}^{I_2(n)}.$
\end{center}
\item Now let $k=2$. Then $\tilde{H}_{n-3}(\Delta_2(C_n)) $ is  the following one-dimensional representation of $I_2(n)$:
\begin{center}
$\sigma\mapsto (-1)^{n-1}=\mathrm{sgn}(\sigma)$
and $\rho\mapsto \begin{cases} 
(-1)^{\frac{n-2}{2}},
& \text{ $n$ even},\\
(-1)^{\frac{n+1}{2}},  
& \text{ $n$ odd}.
\end{cases}$\end{center}
\end{enumerate}
\end{theorem}
\begin{proof} The dihedral group $I_2(n)$ has the presentation 
$\langle \rho, \sigma: \sigma^n=1, \rho^2=1, \rho\sigma=\sigma^{-1}\rho\rangle,$ where $\sigma$ is the long cycle in $\mathfrak{S}_n$ and $\rho$ is an involution.  Since the set of facets $\mathcal{Z}(C_n)$ is invariant under $I_2(n)$,  the right-hand side $Y$ of \Cref{eqn:equiv-hom-kNOT2} is an $I_2(n)$-module. 

Consider Item (1). Let $k$ be even, and let $A^c$ be a facet in $\mathcal{Z}(C_n)$ fixed by $\rho$, as in Part (1) of \Cref{lem:fixed-points-Zk(Cn)}. Since the set $\mathcal{Z}_k(C_n)$, and hence the spaces 
$\tilde{H}(\hat 0, x)\otimes \tilde{H}(x, \hat 1), x\in \mathcal{Z}_k(C_n)$, are permuted by $I_2(n)$,  writing 
$W_{A^c}=\tilde{H}(\hat 0, A^c)\otimes \tilde{H}(A^c, \hat 1)$, we have 
\[Y=\bigoplus_y y\cdot W_{A^c}
\]
where the sum ranges over all coset representatives $y$ of $\langle\rho\rangle$.  \Cref{lem:fixed-points-Zk(Cn)} tells us that $\rho$ acts like 
$\rho\mapsto (-1)^{\frac{n-k}{2}}$ on $\tilde{H}(\hat 0, A^c)$, 
and clearly trivially on $\tilde{H}(A^c, \hat 1)$, so $\rho$ acts on $W_{A^c}$ like $\hat\rho_{n,k}$, and $Y$ is in fact the induced module 
\[\hat\rho_{n,k}\big\uparrow_{\langle \rho\rangle}^{I_2(n)}.\]

When $k$ is odd, we take $B^c$ to be a facet in $\mathcal{Z}(C_n)$ fixed by $\tau$, as in Part (1) of \Cref{lem:fixed-points-Zk(Cn)}, and the result follows mutatis mutandis.

Item (2) follows  similarly using Part (2) of \Cref{lem:fixed-points-Zk(Cn)}.  

Finally consider Item (3). 
Now $\tilde{H}_{n-3}(\Delta_2(C_n))$ is a one-dimensional representation of $I_2(n)$,  so is completely determined by the character values on $\sigma$ and $\rho$, where as above $\rho$ is the fixed-point-free involution if $n$ is even, and an involution with one fixed point if $n$ is odd. 

From \cite[Theorem 7.15]{BDJRSX2024} we have $\sigma\mapsto (-1)^{n-1}$.

Applying \Cref{lem:fixed-points-Zk(Cn)} to \Cref{eqn:trace-kIS2} gives us the  trace computation
\[ (-1)\mathrm{tr}(\rho, \tilde{H}(\Delta_2(C_n)) = \chi^{(2, 1^{n-2})}(\rho)
+(-1)^{n-3}
\begin{cases}
2 (-1)^{\frac{n-2}{2}}, & \text{if $n$ is even},\\
(-1)^{\frac{n-1}{2}},
& \text{if $n$ is odd}. 
\end{cases}
\]
Since by the Murnaghan-Nakayama rule, $\chi^{(2, 1^{n-2})}(\rho)$ equals $(-1)^{\frac{n-2}{2}}$ if $n$ is even, and is  0 if $n$ is odd, the claim follows.
\end{proof}

\section{Conclusion} 

In this paper we described the $f$- and $h$-polynomials for the cut complexes of the disjoint union of graphs in terms of the polynomials for the components.  We showed that cut complexes of squared paths are shellable, and hence homotopy equivalent to a wedge of equidimensional spheres.  We  studied the cut complexes of grid graphs, showing that for $k=3,4$ the homotopy type is a wedge of spheres, in fact even shellable for $k=3$, and computing the Betti numbers, as well as giving partial information for $k=6$.   Finally we  completely determined  the group representation on the rational homology of the cut complex for the disjoint union of complete graphs, and for paths and cycles first considered in \cite{BDJRSX2024}.

In \Cref{conj:Betti-numbers-SqPaths} we propose an enumeratively intriguing  recurrence for the Betti numbers of the (shellable) $k$-cut complexes of squared paths for $4\le k\le n-4$.  
Another open problem, see \Cref{conj:GridGraphs}, is whether the $k$-cut complex of the grid graph $G(m,n)$ is shellable for all $3\le k\le mn-3$.    We would also like to have more information on the Betti numbers of cut complexes for grid graphs.

\section{Acknowledgments}
We thank the organizers of the 2021 Graduate Research Workshop in Combinatorics, where this work originated.  Marija Jeli\'c Milutinovi\'c has been supported by the Project No.\ 7744592 MEGIC ``Integrability and Extremal Problems in Mechanics, Geometry and Combinatorics'' of the Science Fund of Serbia, and by the Faculty of Mathematics University of Belgrade through the grant (No.\ 451-03-47/2023-01/200104) by the Ministry of Education, Science, and Technological Development of the Republic of Serbia. 

\bibliographystyle{plain}
\bibliography{ARXIV-GridfVecs-2024July9-FINAL}
\end{document}